\documentclass[reqno,a4paper]{amsart}

\usepackage{times,enumerate}
\usepackage{latexsym,amsopn,amssymb,amsmath}
\usepackage{amsthm}
\usepackage{amsfonts,amsbsy,amscd}
\usepackage{hyperref}
 \hypersetup{colorlinks=true,linkcolor=Emerald,urlcolor=blue} 
\usepackage[usenames,dvipsnames]{color}
\usepackage[usenames,dvipsnames]{xcolor}

\usepackage{mathrsfs}

\renewcommand{\thefootnote}{\dag}

\newcommand{\C}{\mathbb{C}}

\newcommand{\N}{\mathbb{N}}

\newcommand{\R}{\mathbb{R}}

\newcommand{\ca}{\mathcal{A}}
\newcommand{\cb}{\mathcal{B}}
\newcommand{\cc}{\mathcal{C}}

\newcommand{\ce}{\mathcal{E}}

\newcommand{\cg}{\mathcal{G}}
\newcommand{\ch}{\mathcal{H}}

\newcommand{\ck}{\mathcal{K}}

\newcommand{\cp}{\mathcal{P}}
\newcommand{\cR}{\mathcal{R}}

\def\rond{\mathscr}
\newcommand{\ra}{\rond{A}}
\newcommand{\rb}{\rond{B}}
\newcommand{\rc}{\rond{C}}

\newcommand{\re}{\rond{E}}

\newcommand{\rk}{\rond{K}}

\newcommand{\rr}{\rond{R}}
\newcommand{\rs}{\rond{S}}

\def\rmb{\mathrm{b}}
\def\rmc{\mathrm{c}}

\def\dd{\mathrm{d}}

\def\d{\mathrm{d}}

\def\e{\mathrm{e}}

\def\rmh{\mathrm{h}}
\def\rmi{\mathrm{i}}

\def\veps{\varepsilon}

\def\braket#1#2{\langle{#1}|{#2}\rangle}
\def\jap#1{\langle {#1} \rangle} 
\DeclareMathOperator*{\slim}{s-lim}
\DeclareMathOperator*{\mlim}{m-lim}
\def\cbu{\mathcal{C}_{\mathrm b}^{\mathrm u}}
\def\ess{\mathrm{ess}}
\def\nin{\notin}
\def\pprod{\textstyle \prod}
\def\ccup{\textstyle{\bigcup}}
\def\ccap{\textstyle{\bigcap}}

\long\def\symbolfootnote[#1]#2{\begingroup%
\def\thefootnote{\fnsymbol{footnote}}\footnote[#1]{#2}\endgroup}

\newtheorem{theorem}{Theorem}[section]
\newtheorem{lemma}[theorem]{Lemma}
\newtheorem{proposition}[theorem]{Proposition}
\newtheorem{corollary}[theorem]{Corollary}

\newtheorem{notations}[theorem]{Notations}
\newtheorem{definition}[theorem]{Definition}
\newtheorem{remark}[theorem]{Remark}

\newtheorem{example}[theorem]{Example}

\renewcommand{\hat}{\widehat}
\renewcommand{\tilde}{\widetilde}
\renewcommand{\bar}{\overline}

\newcommand{\RR}{\mathbb{R}}
\newcommand{\CC}{\mathbb{C}}
\renewcommand{\SS}{\mathbb{S}}

\def\maB{\mathcal{B}} 
\def\maC{\mathcal{C}}

\def\maF{\mathcal{F}}

\def\maK{\mathcal{K}}

\newcommand{\oX}{\overline{X}}

\newcommand{\oXY}{\overline{X/Y}}

\newcommand{\maCX}{\maC(\overline{X})}
\newcommand{\maCXY}{\maC(\overline{X/Y})}

\newcommand\ie{{i.\kern2pt e.\ }}

\let\oldtocsection=\tocsection
\let\oldtocsubsection=\tocsubsection
\let\oldtocsubsubsection=\tocsubsubsection
\renewcommand{\tocsection}[2]{\hspace{0em}\oldtocsection{#1}{#2}}
\renewcommand{\tocsubsection}[2]{\hspace{2em}\oldtocsubsection{#1}{#2}}
\renewcommand{\tocsubsubsection}[2]{\hspace{3em}\oldtocsubsubsection{#1}{#2}}

\begin{document}

\title[Essential Spectrum]
{On the Essential Spectrum of N-Body Hamiltonians with Asymptotically
  Homogeneous Interactions
}

\author[V. Georgescu]{Vladimir Georgescu}
     \address{V. Georgescu, D\'epartement de Math\'ematiques,
       Universit\'e de Cergy-Pontoise, 95000 Cergy-Pontoise, France}
     \email{vladimir.georgescu@math.cnrs.fr}

     \thanks{}

     \author[V. Nistor]{Victor Nistor} \address{ D\'epartement de
       Math\'ematiques, Universit\'{e} de Lorraine, 57045 METZ, France
       and Inst. Math. Romanian Acad.  PO BOX 1-764, 014700 Bucharest
       Romania} \email{victor.nistor@univ-lorraine.fr}

\thanks{ V.G. and V.N. have been partially supported by Labex MME-DII
  (ANR11-LBX-0023-01) and ANR-14-CE25-0012-01 respectively.
V.N. manuscripts available from 
{\bf http:{\scriptsize//}iecl.univ-lorraine.fr{\scriptsize /} 
$\tilde{}$ Victor.Nistor{\scriptsize /}}\\
\hspace*{3.7mm} 
AMS Subject classification (2010):   
81Q10 (Primary) 35P05, 47L65, 47L90, 58J40 (Secondary) 
}

\begin{abstract}
  We determine the essential spectrum of Hamiltonians with $N$-body
  type interactions that have radial limits at infinity.  This extends
  the HVZ-theorem, which treats perturbations of the Laplacian by
  potentials that tend to zero at infinity. Our proof involves
  $C^*$-algebra techniques that allows one to treat large classes of
  operators with local singularities and general behavior at
  infinity. In our case, the configuration space of the system is a
  finite dimensional, real vector space $X$, and we consider the
  algebra $\ce(X)$ of functions on $X$ generated by functions of the
  form $v\circ\pi_Y$, where $Y$ runs over the set of all linear
  subspaces of $X$, $\pi_Y$ is the projection of $X$ onto the quotient
  $X/Y$, and $v:X/Y\to\C$ is a continuous function that has uniform
  radial limits at infinity. The group $X$ acts by translations on
  $\ce(X)$, and hence the crossed product $\re(X) := \ce(X)\rtimes X$
  is well defined; the Hamiltonians that are of interest to us are the
  self-adjoint operators affiliated to it.  We determine the
  characters of $\ce(X)$. This then allows us to describe the quotient
  of $\re(X)$ with respect to the ideal of compact operators, which in
  turn gives a formula for the essential spectrum of any self-adjoint
  operator affiliated to $\re(X)$.
\end{abstract}

\maketitle 
\tableofcontents

\section{Introduction}\label{s:intro}

\subsection{}

Let $X$ be a real, finite dimensional vector space and let $X^*$
denote its dual.  If $Y \subset X$ is a subspace, $\pi_{Y} : X \to
X/Y$ will denote the canonical projection.  Let $\ce(X)$ be the
closure in norm of the algebra of functions on $X$ generated by all
functions of the form $u\circ\pi_Y$, where $Y$ runs over the set of
all linear subspaces of $X$ and $u:X/Y\to\C$ runs over the set of
continuous functions that have uniform radial limits at
infinity. Since $X$ acts continuously by translations on $\ce(X)$, we
can define the crossed product $C^{*}$-algebra $\re(X) :=
\ce(X)\rtimes X$, which will be regarded as an algebra of operators on
$L^2(X)$.  Our main result on essential spectra gives a description of
the essential spectrum of any self-adjoint operator $H$ on $L^2(X)$
that is affiliated to $\re(X)$, i.e.\ such that
$(H+\rmi)^{-1}\in\re(X)$. To state this result, we need to first
introduce some notation. Thus, for $x\in X$, we let $T_x$ denote the
translation operator on $L^2(X)$, defined by $(T_x f)(y) := f(y-x)$.
Let $\SS_X$ be the set of half-lines in $X$, that is
\begin{equation}\label{eq:defSX}
 \SS_X \, := \, \{\, \hat{a},\ a\in X, a \neq 0 \,\} \, 
\quad\text{where}\quad \hat{a} := \{ra, \, r>0\}. 
\end{equation}
We let $\overline{\cup} S_\alpha$ denote the {\em closure of the
  union} of a family of sets $S_\alpha$.

\begin{theorem}\label{th:ess-last-intro} 
  If $H$ is a self-adjoint operator affiliated to $\re(X)$, then for
  each $a\in \alpha\in\SS_X$, the limit $\tau_{\alpha}(H) := \alpha.H
  := \slim_{r\to+\infty} T_{ra}^* H T_{ra}$ exists and $\sigma_\ess(H)
  \, = \, \overline{\cup}_{\alpha\in\SS_X} \sigma(\alpha.H)$.
\end{theorem}

For the proof, see Subsection \ref{ss:ham} (Theorem
\ref{th:ess-last}). The meaning of the limit above is discussed in
Remark \ref{re:limit}. Here we note only that it is slightly more
general than the strong resolvent limit since the $\alpha.H$ could be
not densely defined, cf.\ Remark \ref{re:pathology}.

Theorem \ref{th:ess-last-intro} is a consequence of Theorem
\ref{th:recomp} that is our main technical result since it gives a
description of the quotient $C^*$-algebra of $\re(X)$ with respect to
the ideal of compact operators, cf.  Corollary \ref{co:quotient}.
More precisely, let $[\alpha]$ denote the one dimensional linear
subspace generated by $\alpha \in \SS_X$ and let $\ce(X/[\alpha])$ be
the subalgebra of $\ce(X)$ generated by the functions $u\circ\pi_Y$
with $Y\supset\alpha$ and $u$ as before.  Then $\ce(X/[\alpha])$ is
stable under translations, so the crossed product
$\ce(X/[\alpha])\rtimes X$ is well defined, and we have a canonical
embedding
\begin{equation}\label{eq:quotient-intro}
  \re(X)/\rk(X) \, \hookrightarrow \, \pprod_{\alpha\in\SS_X} \,
  \ce(X/[\alpha]) \rtimes X
\end{equation} 
defined as follows. For any $A\in\re(X)$ the limit
$\slim_{r\to+\infty} T_{ra}^*AT_{ra}=: \tau_\alpha(A)$ exists, the map
$\tau_\alpha$ is a $*$-algebra morphism and a linear projection of
$\re(X)$ onto its subalgebra $\re(X/[\alpha])$, and an operator $A
\in\re(X)$ is compact if, and only if, $\tau_\alpha(A)=0$ for all
$\alpha\in\SS_X$.  Then the injective morphism
\eqref{eq:quotient-intro} is induced by the map $\tau(A) :=
\big(\tau_\alpha(A) \big)_{\alpha\in\SS_X}$.

\subsection{}\label{ss:1.2}

In the next few subsections we give some concrete examples of
self-adjoint operators on $L^2(X)$ affiliated to $\re(X)$.

We recall first some facts concerning the spherical compactification
of $X$ (see Section \ref{s:sphcomp}).  The set $\SS_X$ is thought of
as the sphere at infinity of $X$ and $\oX:=X\cup\SS_X$, equipped with
a certain compact space topology, is the spherical compactification of
$X$.  If $f$ is a complex valued function on $X$ and $\alpha\in\SS_X$,
then $\lim_{x\to\alpha} f(x)=c$ (or $\lim_\alpha f=c$) in the sense of
the topology of $\oX$ means the following: ``for any $\varepsilon>0$
there is an open truncated cone $C$ such that $\alpha$ is eventually%
\footnote{ A subset of $X$ is a cone if it is a union of half-lines. A
  truncated cone $C$ is the intersection of a cone with the complement
  of a bounded set. A half-line $\alpha$ is \emph{eventually} in such
  a $C$ if there is $a\in\alpha$ such that $ra\in C\ \forall r>1$.}
in $C$ and $|f(x)-c|<\varepsilon$ if $x\in C$.''  Functions $f$
  with values in an arbitrary topological spaces are treated in
  exactly the same way.

The algebra $\cc(\oX)$ of continuous functions on $\oX$ can be
identified with the set of continuous functions $u$ on $X$ such that
$\lim_\alpha u$ exists for all $\alpha\in\SS_X$ (this is equivalent to
the existence of the uniform radial limits at infinity).  Then we can
regard $\cc(\oXY)$ as an algebra of continuous functions on $X/Y$.
Indeed, when there is no danger of confusion, we will identify a
function $u$ on $X/Y$ with the function $u\circ\pi_Y$ on $X$. Thus, we
shall think of $\cc(\oXY)$ as an algebra of continuous functions on
$X$. Then $\ce(X)$ is the $C^*$-algebra generated by these algebras
when $Y$ runs over the set of all subspaces of $X$.

The following notion of \emph{convergence in the mean} \label{p:mean}
at points $\alpha\in\SS_X$ is natural in our context (see Section
\ref{s:sphcomp}). If $u$ is a complex function on $X$ and $c$ is a
complex number then we write $\mlim_\alpha u=c$, or
$\mlim_{x\to\alpha} u(x)=c$, if $
\lim_{a\to\alpha}\int_{a+\Lambda}|u(x)-c|\dd x=0$ for some (hence any)
compact neighborhood of the origin $\Lambda$ in $X$. This also makes
sense if $u\in L^1_{\mathrm{loc}}(X)$.

Let $\cb(\oX)$ be the set of functions $u\in L^\infty(X)$ such that
$\mlim_\alpha u$ exists for any $\alpha\in\SS_X$. This is clearly a
$C^*$-subalgebra of $L^\infty(X)$. Then $\cb(\oXY)$ is well defined
for any subspace $Y\subset X$ and we have an obvious $C^*$-algebra
embedding $\cb(\oXY)\subset L^\infty(X)$.

Finally, let $\ce^\sharp(X)\subset L^\infty(X)$ be the
$C^*$-subalgebra generated by the algebras $\cb(\oXY)$ when $Y$ runs
over the set of all linear subspaces of $X$. The algebra
$\ce^\sharp(X)$ is, in some sense, a natural extension of $\ce(X)$,
cf.\ Proposition \ref{pr:bgx}.

\subsection{}

We now give the several examples of operators affiliated to
$\re(X)$. Recall that $X^*$ is the vector space dual to $X$.  First,
consider pseudo-differential operators of the form
\begin{equation}\label{eq:Hh}
 H \, = \, h(p) + v
\end{equation}
where $h: X^* \to \R$ is a continuous proper function and
$v\in\ce^\sharp(X)$ is a real function. Here 
\begin{equation}\label{eq.def.hP}
 h(p) \, := \, \maF^{-1} m_h \maF \,,
\end{equation}
where $\maF$ is a Fourier transform $L^2(X) \to L^2(X^*)$ and $m_h$
denotes the operator of multiplication by $h$.  Recall that a function
$h: X^* \to \RR$ is said to be \emph{proper} if $|h(k)|\to\infty$ for
$k\to\infty$. It is clear that the operator $H$ given by \eqref{eq:Hh}
is self-adjoint on the domain of $h(p)$, since $h(p)$ is self-adjoint
by spectral theory and $v$ is a bounded operator.

As a second example of affiliated operators, we consider differential
operators on $X=\R^n$ of the form
\begin{equation}\label{eq:L}
L \, = \, \sum_{|\mu|,|\nu|\leq m} p^\mu g_{\mu\nu}p^\nu
\end{equation} 
where $m\geq1$ is an integer and $g_{\mu\nu}\in\ce^\sharp(X)$. The
notations are standard: $p_j=-\rmi\partial_j$, where $\partial_j$ is
the derivative with respect to the $j$-th variable, and for
$\mu=(\mu_1,\dots,\mu_n)\in\N^n$ we set $p^\mu=p_1^{\mu_1}\dots
p_n^{\mu_n}$ and $|\mu|=\mu_1+\dots+\mu_n$.  For real $s$, let $\ch^s$
be the usual Sobolev space on $X$, in particular $\ch^0
\equiv\ch=L^2(X)$.  Then $L:\ch^m\to\ch^{-m}$ is a well defined
operator and we assume that there exist $\gamma,\delta>0$ such that
\begin{equation}\label{eq:L1}
  \braket{u}{Lu} + \gamma\|u\|^2 \, \geq \, \delta \|u\|_m^2\,, \quad
 \text{for all } \  u\in\ch^m. 
\end{equation}
Here $\|\cdot\|$ and $\|\cdot\|_m$ denote the usual norms on $\ch$ and
$\ch^m$.  Note that, since the $g_{\mu\nu}$ are bounded, this is a
condition only on the principal part of $L$ (i.e. the part
corresponding to $|\mu|=|\nu|=m$). Then $L+\gamma:\ch^m\to\ch^{-m}$ is
a symmetric isomorphism, and hence the restriction of $L$ to
$(L+\gamma)^{-1}\ch$ is a self-adjoint operator in $\ch$ that we will
denote by $H$: 
\begin{equation}\label{eq.second.H}
 H \, := \, L : (L+\gamma)^{-1}\ch \, \to \, \ch\,.
\end{equation}

\begin{theorem}\label{th:intro1} 
  Both operators $H$ defined above in Equations \eqref{eq:L} and 
  \eqref{eq.second.H} are affiliated to $\re(X)$.
\end{theorem}

We make some remarks in connection with the Theorem \ref{th:intro1}.
\begin{enumerate}

\item If $H$ is given by \eqref{eq:Hh} then the strong limit in
  Theorem \ref{th:ess-last-intro} exists in the usual sense of
  pointwise convergence on the domain of $H$. (Let us notice that in
  this case, the domain of $H$ is invariant for the action of
  $T_{ra}$.)  If $H$ is associated to the operator $L$ from
  \eqref{eq:L}, then the limit holds in the strong topology of
  $B(\ch^m,\ch^{-m})$.

\item The union that gives $\sigma_\ess(H)$ in Theorem
  \ref{th:ess-last-intro} may contain an infinite number of distinct
  terms even in simple $N$-body type cases. Indeed, an example can be
  obtained by choosing $X=\R^2$,
  $\mathcal Y$ to be a countable set of lines (whose union could be dense in
  $X$), and $H := \Delta+\sum_Y v_Y$ for some conveniently chosen
  $v_Y\in\cc_\rmc(X/Y)$ satisfying $\sum_Y\sup|v_Y|<\infty$.

\item The coefficients $g_{\mu\nu}$ in the principal part of $L$ are
  bounded Borel functions, and locally this cannot be improved.  But
  the other coefficients $g_{\mu\nu}$ and the potential $v$ are
  assumed bounded only for the sake of simplicity, see Remark
  \ref{re:unbcoeff} for more general results. Later on, we shall also treat
  unbounded, not necessarily local perturbations. See for example
  Theorems \ref{th:ess4-intro} and \ref{th:hconj-intro}.

\item We stated the applications of the abstract theorems in a way
  adapted to elliptic operators, but the extension to hypoelliptic
  operators is easy: it suffices to consider functions $h\in C^m$ with
  bounded derivatives of order $m$ and to replace the Sobolev spaces
  by spaces associated to weights of the form
  $\sum_{|\mu|\leq m} |h^{(\mu)}(k)|$.

\item Theorem \ref{th:intro1} and the other results of the same nature
  remain true, with essentially no change in the proof, if the space
  $L^2(X)$ is replaced by $L^2(X)\otimes E$ with $E$ a finite
  dimensional complex Hilbert space and $g_{\mu\nu}$ and $v$ are
  $\cb(E)$-valued functions.  For this it suffices to work with the
  algebra $\re(X)\otimes \cb(E)$ or the more general and natural
  object $\re(X)\otimes \ck(E)$ where $E$ can be an infinite
  dimensional Hilbert space. This covers matrix differential
  operators, e.g. the Dirac operator, which are not semi-bounded, and
  hence the general affiliation criterion Theorem \ref{th:recall0} has
  to be used.

\end{enumerate}

Operators of the form \eqref{eq:Hh} (and hence also Theorem
\ref{th:intro1} and its generalizations) cover many of the most
interesting (from a physical point of view) Hamiltonians of $N$-body
systems.  Here are two typical examples. First, in the
non-relativistic case, $X$ is equipped with a Euclidean structure and
a typical choice for $h$ is $h(\xi) = |\xi|^2$, which gives
$h(p) =\Delta$.  Second, in the case of $N$ relativistic particles of
spin zero and masses $m_1,\dots,m_N$, we take $X=(\R^3)^N$ and,
writing the momentum $p$ as $p=(p_1,\dots,p_N)$ where
$p_j=-\rmi\nabla_j$ acts in $L^2(\R^3)$, we have
$h(p)=\sum_{j=1}^N(p_j^2+m_j^2)^{1/2}$. We refer to
\cite{Derezinski-Gerard} for a thorough study of the spectral and
scattering theory of the non-relativistic $N$-body Hamiltonians with
$k$-body potentials that tend to zero at infinity.  We also note that
second order perturbations with an $N$-body structure of the
Laplacian, i.e.\ operators $L$ of second order with non trivial
$g_{\mu\nu}$ in the principal part, are of physical interest in the
context of pluristratified media \cite{DI}.

\subsection{}

The structure of the potential $v$ and of the coefficients $g_{\mu\nu}$ considered
\label{page:nbody}
above is more complicated than in the usual case of $N$-body hamiltonians because it
can contain products of the form $v_E\circ\pi_E\cdot v_F\circ\pi_F$,
which cannot be written as $v_G\circ\pi_G$ as in the usual $N$-body
case (here $E,F,G$ are subspaces of $X$ and in the usual $N$-body situation
one may take $G=E\cap F$).  If $v$ has a simpler structure, similar to
that of the standard $N$-body potentials, then Theorem \ref{th:intro1}
may be reformulated in a way that stresses the similarity with the
usual HVZ theorem.  Moreover, in this case we will be able to treat a
considerably more general class of nonlocal potentials $v_Y$ (see
Subsection \ref{ss:nonloc-intro}).

Assume that, for each subspace $Y \subset X$, a real function
$v_Y\in\cb(\oXY)$ is given such that $v_Y=0$ for all but a finite
number of subspaces $Y$ and let $v=\sum_Y v_Y\in\ce^\sharp(X)$ (recall
the identification $v_Y \equiv v_Y\circ\pi_Y$). If
$\alpha\not\subset Y$ then $\pi_Y(\alpha) \in\SS_{X/Y}$ is a well
defined half-line in the quotient and we may define
$ v_Y(\alpha)=\mlim_{\pi_Y(\alpha)} v_Y$ (see Subsection \ref{ss:1.2}
page \pageref{p:mean}).

\begin{proposition}\label{pr:intro2}
  For each $\alpha\in\SS_X$, let
\begin{equation}\label{eq:intro2}
  H_\alpha\, := \, h(p)
  + \sum_{Y\supset\alpha} v_Y +
  \sum_{Y\not\supset\alpha} v_Y(\alpha) \, .
\end{equation}
Then $\alpha.H = H_{\alpha}$ and hence
$\sigma_\ess(H) =\overline{\cup}_{\alpha\in \SS_{X}}
\sigma(H_{\alpha})$.
\end{proposition}

\begin{remark}\label{re:usual}{\rm The usual $N$-body type
    Hamiltonians are characterized by the condition that all the
    $v_Y:X/Y\to\R$ vanish at infinity. Then we obtain
    $\alpha.H=h(P)+\sum_{Y\supset\alpha}v_Y$, so Proposition
    \ref{pr:intro2} becomes the usual version of the HVZ theorem.
  }\end{remark}

\subsection{}\label{ss:nonloc-intro}

We give now further examples of self-adjoint operators with nonlocal
and unbounded interaction affiliated to $\re(X)$.  For $k\in X^*$ let
$M_k$ be the multiplication operator defind by
$(M_k f)(x)=\e^{\rmi k(x)} f(x)$. Later we shall use the notation
$\braket{x}{k}:=k(x)$.

If the perturbation $V$ is bounded, then there are no restrictions on
$h$ besides being proper and continuous. Indeed, we have the 
following result.

\begin{theorem}\label{th:ess3-intro}
  Let $H=h(p)+V$, where $h:X^*\to\R$ is a continuous, proper
%
%
  function and $V=\sum_Y V_Y$ is a finite sum with $V_Y$ bounded
  symmetric linear operators on $L^2(X)$ satisfying:
\begin{enumerate}[(i)]
\item $\lim_{k\to0} \| [M_k, V_Y] \|=0$,
\item $[T_y,V_Y]=0$ for all $y\in Y$,
\item $\slim_{a\in X/Y,a\to\alpha} T_a^* V_Y T_a$ exists for each
  $\alpha\in\SS_{X/Y}$.
\end{enumerate}
Then $H$ is affiliated to $\re(X)$.
\end{theorem}

Note that in the above theorem \ref{th:ess3-intro},
the operator $T_x^*V_YT_x$ depends a priori on the point $x$ in $X$, but if
condition (ii) is satisfied, then it depends only on the class
$\pi_Y(x)$ of $x$ in $X/Y$. Therefore we may set
$T_{\pi_Y(x)}^*V_YT_{\pi_Y(x)}=T_x^*V_YT_x$, which gives a meaning to
$T_a^*V_YT_a$ for any $a\in X/Y$ in condition (iii) above.

In order to treat unbounded interactions, we have to require more
regularity on the function $h$. We denote by $|\cdot|$ a Euclidean norm
on $X^*$.

\begin{theorem}\label{th:ess4-intro}
  Let $h:X^*\to[0,\infty)$ be locally Lipschitz with derivative $h'$
  such that for some real numbers $c,s>0$ and all $k\in X^*$ with
  $|k|>1$
\begin{equation}\label{eq:ess4-intro}
  c^{-1}|k|^{2s} \, \leq \, h(k) \, \leq \, c |k|^{2s} 
  \quad\text{and} \quad |h'(k)| \, \leq \, c |k|^{2s} .
\end{equation}
Let $V=\sum V_Y$ be a finite sum with $V_Y:\ch^s\to\ch^{-s}$ symmetric
operators satisfying:
\begin{enumerate}[(i)]
\item there are numbers $\gamma,\delta$ with $\gamma<1$ such that
  $V\geq -\gamma h(p) -\delta$,
\item $ \lim_{k\to 0}\| [M_k, V_Y] \|_{\ch^s\to\ch^{-s}}=0$,
\item $[T_y,V_Y]=0$ for all $y\in Y$,
\item $\slim_{a\in X/Y,a\to\alpha} T_a^* V_Y T_a$ exists in
  $B(\ch^s,\ch^{-s})$ for all $\alpha\in\SS_{X/Y}$.
\end{enumerate}
Then $h(p)+V$ is a symmetric operator $\ch^s\to\ch^{-s}$, which
induces a self-adjoint operator $H$ in $L^2(X)$ affiliated to
$\re(X)$.
\end{theorem}

\begin{remark}\label{ex:fct}{\rm If $V_Y$ is the operator of
    multiplication by a measurable function, then Condition (ii) of
    Theorem \ref{th:ess4-intro} is automatically satisfied. On the
    other hand, Condition (iii) gives that $V_Y(x+y)=V_Y(x)$ for all
    $x\in X$ and $y\in Y$. This means that $V_Y=v_Y\circ\pi_Y$ for a
    measurable function $v_Y:X/Y\to\R$, which has to be such that the
    operator of multiplication by $V_Y$ is a continuous map
    $\ch^s(X)\to\ch^{-s}(X)$. For this it suffices that the operator
    $v_Y(q_Y)$ of multiplication by $v_Y$ be a continuous map of
    $\ch^s(X/Y)$ into $\ch^{-s}(X/Y)$ (here $q_Y$ is the position
    observable in $L^2(X/Y)$). Then the last condition means that
    $\lim_{a\to\alpha} v_Y(q_Y+a)$ exists strongly in
    $\cb\big(\ch^s(X/Y),\ch^{-s}(X/Y)\big)$.  }
\end{remark}

\begin{example}\label{ex:5}{\rm Let us consider the case of
    non-relativistic Schr\"odinger operators. Then $X$ is a Euclidean
    space (so we may identify $X/Y=Y^\perp$) and $H_0:=\Delta=p^2$ is
    the (positive) Laplace operator, and hence $s=1$. The total
    Hamiltonian is of the form $H=\Delta+\sum_Y V_Y$ where the sum is
    finite and $V_Y=1\otimes V_Y^\circ$ where
    $V_Y^\circ:\ch^1(Y^\perp)\to\ch^{-1}(Y^\perp)$ is a symmetric
    linear operator whose relative form bound with respect to the
    Laplace operator on $Y^\perp$ is zero (this is much more than we need).
    Then assume $M_kV_Y^\circ=V_Y^\circ M_k$ for all $k\in Y^\perp$.
    For example, $V_Y^\circ$ could be the operator of multiplication
    by a function $v_Y:Y^\perp\to\R$ of Kato class $K_{n(Y)}$ with
    $n(Y)=\dim(Y^\perp)$ (see Section 1.2 in \cite{CFKS}, especially
    assertion (2) page 8) but it could also be a distribution of non
    zero order. Indeed, we may take as $v_Y$ the divergence of a
    vector field on $Y^\perp$ whose components have squares of Kato
    class (e.g.\ are bounded functions): this covers highly
    oscillating perturbations of potentials that have radial limits
    at infinity.  Note that this Kato class is convenient because then
    $v_Y\circ\pi_Y$ is of class $K_{\dim(X)}$, see \cite[p.\
    8]{CFKS}. To get (iv) of Theorem \ref{th:ess4-intro} it suffices
    to assume $\lim_{a\to\alpha}v_Y(\cdot+a)$ exists strongly in
    $B(\ch^1(Y^\perp),\ch^{-1}(Y^\perp))$ for each
    $\alpha\in\SS_{Y^\perp}$.  }\end{example}

\subsection{}

The \emph{spherical algebra}
$\rs(X) := \maCX \rtimes X \subset \re(X)$ has several interesting
properties. For example, it contains the ideal $\rk(X)=\cc_0(X)\rtimes X$
of compact operators on $L^2(X)$. It is remarkable that both $\rs(X) $
and its quotient $\rs(X)/\rk(X)$ may be described in quite explicit
terms. In the next theorem and in what follows, we adopt the following
convention: if we write $S^{(*)}$ in a relation, then it means that
that relation holds for $S^{(*)}$ replaced by either $S$ or $S^*$. Let
$C^*(X)$ be the group $C^*$-algebra of $X$, cf. Section
\ref{s:locinf}.

\begin{theorem}\label{th:sphalg-intro}
  A bounded operator $S$ on $L^2(X)$ belongs to $\rs(X)$ if, and only if,
\begin{equation*}
 \begin{gathered}
  \lim_{x\to0}\| (T_x-1)S^{(*)} \|=0 \,,\quad\lim_{k\to0}\| [M_k, S]\|=0 \,, 
  \quad \mbox{and} \\  
 \slim_{a\to\alpha} T_a^* S^{(*)} T_a \quad\text{exists for any }\ \
  \alpha\in\SS_X \,.
 \end{gathered}
\end{equation*}
If $S\in \rs(X)$ and $\alpha\in\SS_X$, then
$\tau_\alpha(S)=\slim_{a\to\alpha} T_a^* S T_a$ belongs to
$C^*(X)$. The map $\tau(S):\alpha\mapsto\tau_\alpha(S)$ is norm
continuous, so $\tau: \rs(X) \to C(\SS_X)\otimes C^*(X)$. This map
$\tau$ is a surjective morphism and its kernel is $\rk(X)$.  Hence we
have a natural identification
\begin{equation}\label{eq:quot-intro}
  \rs(X) /\rk(X) \, \cong \, \cc(\SS_X)\otimes C^*(X) \, \cong \,
  \cc_0(\SS_X \times X\sp{*}) \, .
\end{equation}
If $H$ is a self-adjoint operator affiliated to $\rs(X)$, then the
limit $\alpha.H:=\slim_{a\to\alpha} T_a^*HT_a $ exists 
for
each $\alpha\in\SS_X$ 
and
$\sigma_\ess(H) = \cup_\alpha \sigma(\alpha.H)$.
\end{theorem}

Note that in this theorem (as well as in the next), we consider the
plain union, not its closure.  The next result is a general criterion
of affiliation to $\rs(X)$.

\begin{theorem}\label{th:hconj-intro}
Let $H$ be a bounded from below self-adjoint operator on $L^2(X)$ such
that its form domain $\cg$ satisfies the following condition: the
operators $T_x$ and $M_k$ leave $\cg$ invariant, the operators $T_x$
are uniformly bounded in $\cg$, and $\lim_{x\to
  0}\|T_x-1\|_{\cg\to\ch}=0$.  Assume that
$\|[M_k,H]\|_{\cg\to\cg^*}\to0$ as ${k\to0}$ and that the limit
$\alpha.H:=\lim_{a\to\alpha} T_a^* H T_a$ exists strongly in
$\cb(\cg,\cg^*)$, for all $\alpha\in\SS_X$.  Then $H$ is affiliated to
$\rs(X)$, for each $\alpha\in\SS_X$ the operator in $L^2(X)$
associated to $\alpha.H$ is self-adjoint, and
$\sigma_\ess(H) = \cup_\alpha \sigma(\alpha.H)$.
\end{theorem}

Let us notice an important difference between the morphisms
$\tau_{\alpha}$ of Theorems \ref{th:ess-last-intro} and
\ref{th:hconj-intro}.  (These morphisms appear in the computation of
the quotients in \eqref{eq:quotient-intro} and \eqref{eq:quot-intro}.)
More precisely, in the definition of the first $\tau$, we take limits
over $ra$, with $r \to \infty$, that is limits along rays, whereas in
the second one we take general limits $a \to \alpha$ (not just along
the ray $\alpha$). The stronger assumptions in Theorem
\ref{th:hconj-intro} then lead to a stronger result (in that we do not
need the closure of the union to obtain the spectrum).

\subsection{}

Descriptions of the essential spectrum of various classes of
Hamiltonians in terms of limits at infinity of translates of the
operators have already been obtained before, see for example
\cite{Helffer-Mohamed,GI2,RochBookLimit,LastSimon,RochBookNGT} (in
historical order). Our approach is based on the ``localization at
infinity'' technique developed in \cite{GI2,GI3} in the context of
crossed-products of $C^*$-algebras by actions of abelian locally
compact groups. This has been extended to noncommutative unimodular
amenable locally compact groups in \cite{G}, cf.  Proposition 6.5 and
Theorem 6.8 there. The case of noncommutative groups has also been
considered recently in \cite{MantoiueGGroups}.  See \cite{CFKS,Simon4}
for a general introduction to the basics of the problems studied here.

A homogeneous potential of degree zero outside of a compact set models
a force that is perpendicular to the line joining the particle to the
origin, and hence trying to force the particle to move on a sphere.
Results on operators with homogeneous potentials or similar potentials
were obtained, for example, in
\cite{HMV2004,HerbstSkibsted2002,HerbstSkibsted2004,SRichard2005},
where further physical motivation is provided.

In fact, our results shed some new light even on the classical case
when the auxiliary functions $v_Y$ that define $V_Y$ converge to 0 at
infinity, since in our case the spectra of the relevant algebras are
easier to compute and then can be used to describe the spectra in the
classical case. Compared to the classical approach \cite{CFKS, Simon4}
to the essential spectrum of the $N$-body problem, our approach has
the advantage that is more conceptual, and, once a certain machinery
has been developed, one can obtain rather quickly generalizations of
these results to other operators.  It also takes advantage of a rather
well developed theory of crossed products and representations of
$C^{\ast}$-algebras.  We also develop general techniques that may be
useful for the study of other types of operators and of other types of
questions, such as the study of the eigenvalues and eigenfunctions of
$H$, even in the case when the radial limits at infinity are zero.  We
mention that, by using the expression \eqref{eq:intro2} for the
asymptotic operators, one could prove the Mourre estimate as in
\cite[Sec.\ 9.4]{ABG} for the larger class of Hamiltonians considered
in Proposition \ref{pr:intro2}.

\subsection{}\label{ss:cont}

Let us briefly describe the contents of the paper. In Section
\ref{s:locinf}, we recall some facts concerning crossed products with
$X$ of translation invariant $C^*$-algebras of bounded uniformly
continuous functions on $X$ and the role of operators with the
``position-momentum limit property'' in this context. Then we discuss
the question of the computation of the quotient with respect to the
compacts of such crossed products. In Section \ref{s:sphcomp}, we
briefly describe the topology and the continuous functions on the
spherical compactification $\oX$ of a real vector space $X$. This
allows us to introduce and study in Section \ref{s:sphalg} the
spherical algebra $\rs(X) := \maCX \rtimes X$. We obtain an explicit
description of the operators that belong to $\rs(X)$ (Theorem
\ref{th:sphalg}) and we also give an explicit description of the
quotient $\rs(X)/\rk(X)$ (Theorem \ref{th:ker}). The canonical
composition series of this algebra leads to Fredholm conditions, and
hence to a determination of the essential spectrum for the operators
affiliated to it.  In Section \ref{s:affS}, we give some general
criteria for a self-adjoint operator to be affiliated to a general
$C^*$-algebra and apply them to the case of $\rs(X)$.  The algebras
$\ce(X)$ and $\re(X)$ are studied in Section \ref{s:nbd}.  Subsections
\ref{ss:prelim} and \ref{ss:ham} contain the main technical results.
At a technical level, the main result in this section is the
description of the spectrum of $\ce(X)$ (Theorem
\ref{th:char}). Subsection \ref{ss:ham} is devoted to the study of the
Hamiltonian algebra $\re(X)$ and we prove in this section two of our
main results, Theorems \ref{th:recomp} and \ref{th:ess-last}. In
Subsection \ref{ss:affnb} we prove Theorems \ref{th:intro1},
\ref{th:ess3-intro} and \ref{th:ess4-intro}, which describe a general
class of operators affiliated to $\re(X)$ for which we obtain explicit
descriptions of the essential spectrum. We note that Theorem
\ref{th:sphalgY} gives descriptions of the algebras
$\maC(\oXY)\rtimes X$ generating $\re(X)$ that are not relying on
their definition as crossed products.

This paper contains the full proofs of the results announced in
\cite{GN}, as well as several extensions of those results.

\subsection{Acknowledgments} 
We thank B. Ammann and N. Prudhon for useful discussions.  We also
thank the referee for carefully reading the paper and for useful
comments.

\section{Crossed products and localizations at
  infinity}\label{s:locinf}
\protect\setcounter{equation}{0}

In this section, we review some needed results from \cite{GI3}
relating essential spectra of operators and the spectrum (or character
space) of some algebras. 

If $X$ is a finite dimensional vector space, we denote by
$\maC_\rmb(X)$ the algebra of bounded continuous functions on $X$, by
$\maC_0(X)$ its ideal consisting of functions vanishing at infinity,
and by $\cbu(X)$ the subalgebra of bounded uniformly continuous
functions. Let $\rb(X) :=\maB(L^2(X))$ be the algebra of bounded
operators on $L^2(X)$ and $\rk(X) := \maK(L^2(X))$ the ideal of
compact operators.

If $Y$ is a subspace of $X$, we identify a function $u$ on $X/Y$ with
the function $u\circ\pi_Y$ on $X$. In other terms, we think of a
function on $X/Y$ as being a function on $X$ that is invariant under
translations by elements of $Y$. This clearly gives an embedding
$\cbu(X/Y)\subset\cbu(X)$.  The subalgebras of $\cbu(X/Y)$ can then be
thought of as subalgebras of $\cbu(X)$. Thus $\maC_0(X/Y)$ and the
algebra $\maCXY$ that we shall introduce below are both embedded in
$\cbu(X)$.

For any function $u$, we shall denote by $m_u$ the operator of
multiplication by $u$ on suitable $L\sp{2}$ spaces.  If $u:X\to\C$ and
$v:X^*\to\C$ are measurable functions, then $u(q)$ and $v(p)$ are the
operators on $L^2(X)$ defined as follows: $u(q) = m_u$, the
multiplication operator by $u$, and $v(p)= \maF^{-1}m_v\maF$, where
$\maF$ is the Fourier transform $L^2(X) \to L^2(X^*)$.  If $x\in X$
and $k\in X^*$, then the unitary operators $T_x$ and $M_k$ are defined
on $L^2(X)$ by
\begin{equation}\label{eq:trans.mult}
  (T_x f)(y) \, := \, f(y-x) \quad\text{ and } \quad (M_k f)(y) \, :=
  \, \e^{\rmi\braket{y}{k}} f(y) \,,
\end{equation}
and can alternatively be written in terms of $p$ and $q$ as
$T_x=\e^{-\rmi xp}$ and $M_k=\e^{\rmi kq}$.

We shall denote by $C^*(X)$ the {\em group $C\sp{*}$-algebra} of $X$:
this is the closed subspace of $\rb(X)$ generated by the operators of
convolution with continuous, compactly supported functions. The map
$v\mapsto v(p)$ establishes an isomorphism between $\maC_0(X^*)$ and
$C^*(X)$. 

We shall need the following general result about commutative
$C\sp{*}$-algebras. Let $\ca$ be a commutative $C\sp{*}$-algebra and
$\hat\ca$ be its spectrum (or character space), consisting of non-zero
algebra morphisms $\chi : \ca \to \CC$. If $\ca$ is unital, then
$\hat\ca$ is a compact topological space for the weak topology. In
general, it is locally compact and the Gelfand transform
$\Gamma_\ca : \ca \to \maC_0(\hat\ca)$,
$\Gamma_\ca(u)(\chi) := \chi(u)$, defines an isometric algebra
isomorphism. In particular, any commutative $C\sp{*}$-algebra is of
the form $\maC_0(\Omega)$ for some locally compact space (up to
isomorphism). The characters of $\maC_0(\Omega)$ are of the form
$\chi_\omega$, $\omega \in \Omega$, where
\begin{equation}
 \label{eq.def.chi}
 \chi_\omega(u) \, := \, u(\omega)\, \quad u \in \maC_0(\Omega)\,.
\end{equation}
If $X$ acts continuously on a $C\sp{*}$-algebra $\ca$ by
automorphisms, we shall denote by $\ca \rtimes X$ the resulting crossed product
algebra, see \cite{Pedersen, WilliamsBook}.  Here the real vector
space $X$ is regarded as a locally compact, abelian group in the
obvious way.  Recall \cite{GI2} that if $\ca$ is a translation
invariant $C^*$-subalgebra of $\cbu(X)$, then an isomorphic
realization of the cross-product algebra $\ca \rtimes X$ is the norm
closed linear subspace of $\rb(X)$ generated by the operators of the
form $u(q)v(p)$, where $u\in \ca$ and $v \in \maC_0(X\sp{*})$. As a
rule, we shall denote by $\tau_a$ the action of $a \in X$ by translations on our
algebras of functions.

\begin{definition} \label{df:pmlp} 
Let $A \in \rb(X)$. We say that $A$ {\em has the position-momentum
  limit property}\ if $\lim_{x\to0}\| (T_x-1)A^{(*)} \|=0$ and $
\lim_{k\to0}\| [M_k, A] \|=0$.
\end{definition}

A characterization of operators having the position-momentum limit
property in terms of crossed products was given in \cite{GI2}: it is
shown that \emph{$A$ has the position-momentum limit property if, and
  only if, $A\in\cbu(X)\rtimes X$}.

If $A$ is an operator on $L^2(X)$, then its translation by $x\in X$ is
defined by the relation
\begin{equation}\label{eq:str}
  \tau_x(A) \, := \, T_x^*AT_x \,.
\end{equation}
The notation $x.A := \tau_x(A)$ will often be more convenient.  If $u$
is a function on $X$ we also denote $\tau_x(u)\equiv x.u$ its
translation given by $(x.u)(y)=u(x+y)$. The notations are naturally
related: $\tau_x(u(q))=(x.u)(q)\equiv u(x+q)$. Note that
$\tau_x(v(p))=v(p)$.

By a ``point at infinity'' of $X$, we shall mean a point in the boundary
of $X$ in a certain compactification of it. We shall next define
\emph{the translation by a point at infinity $\chi$} for certain
functions $u$ and operators $A$.  This construction will be needed for
the description of the essential spectrum of operators of interest for
us.

Let us fix a translation invariant $C^*$-algebra $\ca$ of bounded
uniformly continuous functions on $X$ containing the functions that
have a limit at infinity: $\maC_0(X)+\C\subset\ca\subset\cbu(X)$.  To
every $x\in X$, there is associated the character $\chi_x$, defined by
$\chi_x(u) := u(x)$ for $u\in\ca$, cf.\ \eqref{eq.def.chi}.  Since
$\ca\supset\maC_0(X)$, $X$ is naturally embedded as an open dense
subset in $\hat{\ca}$.  Thus $\hat\ca$ is a compactification of $X$
and
\begin{equation}\label{eq.def.delta}
 \delta(\ca) \, := \, \hat\ca\setminus X \,,
\end{equation}
the {\em boundary} of $X$ in this compactification, is a compact set
that can be characterized as the set of characters $\chi$ of $\ca$
whose restriction to $\maC_0(X)$ is equal to zero.

Let us recall that if $x, y \in X$, then $(x.u)(y) = u(x+y) =
\chi_x(y.u)$.  If $u \in \ca$, we extend the definition of $x.u$ by
replacing in this relation $\chi_x$ with a character $\chi \in
\hat\ca$.

\begin{definition}\label{def:trinf}
  Let $u \in \ca$ and $\chi \in \hat\ca$. Then we define
  $(\chi.u)(y) :=\chi(y.u)$ for all $y \in X$.
\end{definition}

Since $u$ is uniformly continuous, it is easy to check that
$\tau_\chi(u) := \chi.u \in \cbu(X)$ and that
$\tau_\chi : \ca\to\cbu(X)$ is a unital morphism. We will say that
$\tau_\chi$ is the \emph{morphism associated to the character $\chi$}.
We note that if the character $\chi$ corresponds to $x \in X$, then
$\tau_\chi = \tau_x$, so our notation is consistent.

In particular, we get ``translations at infinity'' of $u\in\ca$ by
elements $\chi\in\delta(\ca)$. The function $\chi\mapsto \chi.u\in
\cbu(X)$ defined on $\hat\ca$ is continuous if $\cbu(X)$ is equipped
with the topology of local uniform convergence, and hence $\chi.u =
{\lim_{x\to\chi}}\, x.u$ in this topology for any
$\chi\in\delta(\ca)$. One has $u\in\maC_0(X)$ if, and only if,
$\chi.u=0$ for all $\chi\in\delta(\ca)$. We mention that a
translation $\chi.u$ by a point at infinity $\chi\in\delta(\ca)$ does
not belong to $\ca$ in general. However, we shall see that this is
true in the case $\ca=\ce(X)$ of interest for us, so in this case
$\tau_\chi$ is an endomorphism of $\ca$.

If $A\in\ca\rtimes X$, then we may also consider ``translations at
infinity'' $\tau_\chi(A)$ by elements $\chi$ of the boundary
$\delta(\ca)$ of $X$ in $\hat\ca$ and we get a useful characterization
of the compact operators.  The following facts are proved in
\cite[Subsection 5.1]{GI3}.

\begin{proposition}\label{pr:locinf}
For each $\chi\in\hat\ca$, there is a unique morphism
$\tau_\chi:\ca\rtimes X\to\cbu(X)\rtimes X$ such that
\begin{equation*}
   \tau_\chi(u(q) v(p) ) \, = \, (\chi.u)(q)v(p) \,, \quad\text{for
     all } u\in\ca, \, v\in\maC_0(X) \, .
\end{equation*}
If $A\in\ca\rtimes X$, then $\chi\mapsto \tau_\chi(A)$ is a strongly
continuous map $\hat\ca\to \rb(X)$.
\end{proposition}

As before, we often abbreviate $\tau_\chi(A)=\chi.A$. This gives a
meaning to the translation by $\chi$ of any operator $A\in\ca\rtimes
X$ and any character $\chi\in\hat\ca$. Observe that $\chi\mapsto
\chi.A$ is just \emph{the continuous extension to $\hat\ca$ of the
  strongly continuous map $X\ni x\mapsto x.A$}. In particular,
\begin{equation}\label{eq:lim}
  \tau_\chi(A) \, = \, \slim_{x\to\chi} T_x^*AT_x \quad\text{ for all
  }\ A\in\ca\rtimes X\ \text{ and } \ \chi\in\delta(\ca) \,.
\end{equation}
We have $\rk(X)=\maC_0(X)\rtimes X\subset\ca\rtimes X$.  Then
\cite[Theorem 1.15]{GI3} gives:

\begin{theorem}\label{th:infty} 
An operator $A \in \ca\rtimes X$ is compact if, and only if,
$\tau_\chi(A)=0 \ \forall \chi\in\delta(\ca)$. In other terms:
$\cap_{\chi\in\delta(\ca)}\ker\tau_\chi=\rk(X)$.  The map
$\tau(A)=(\tau_\chi(A))_{\chi\in\delta(\ca)}$ induces an injective
morphism
\begin{equation}\label{eq:crk}
  	\ca\rtimes X/\rk(X) \hookrightarrow
        \pprod_{\chi\in\delta(\ca)} \cbu(X)\rtimes X .
\end{equation}
\end{theorem}

\begin{remark}\label{re:cross}{\rm
We emphasize the relation between this result and some facts from the
theory of crossed products. The operation of taking the crossed
product by the action of an amenable group transforms exact sequences
in exact sequences \cite[Proposition 3.19]{WilliamsBook}, and hence we have
an exact sequence
\begin{equation}
  0 \to \maC_0(X) \rtimes X \to\ca\rtimes X \to (\ca/\maC_0(X))\rtimes
  X \to 0\,.
\end{equation}
Since $\maC_0(X) \rtimes X \simeq \rk(X)$, we get
$\ca\rtimes X/\rk(X) \simeq\big(\ca/\maC_0(X)\big)\rtimes X$, which
reduces the computation of the quotient $\ca\rtimes X/\rk(X)$ to the
description of $\ca/\maC_0(X)$. This is convenient since 
$\ca/\maC_0(X) \simeq \maC(\delta(\ca))$. Moreover, we have
$\tau_\chi=\tau_\chi\rtimes\mathrm{id_X}$, where the morphisms
$\tau_\chi$ on the right hand side are those corresponding to $\ca$.
We complete this remark by noticing that if $\chi$ and $\chi_1$ are
obtained from each other by a translation by $x \in X$, then the
corresponding morphisms $\tau_{\chi}$ and $\tau_{\chi_1}$ are
unitarily equivalent by the unitary corresponding to $x$. In
particular, in the above theorem and in the following corollary, it
suffices to use one $\chi$ from each orbit of $X$ acting on
$\delta(\ca)$. }
\end{remark}

\begin{remark} {\rm Let us notice that in view of the results in
    \cite{Echterhoff96, WilliamsBook}, the above theorem provides
    nontrivial information on the cross-product algebra
    $\maC(\delta(\ca)) \rtimes X$, and hence on the action of $X$ on
    $\delta(\ca)$. It would be interesting to study the corresponding
    properties for a general Lie group $G$ acting on $\cbu(G)$
    \cite{MantoiueGGroups}.  Morphisms analogous to the $\tau_\chi$
    can be defined also in a groupoid framework \cite{LMN1,
      frascatti}, but they do not have a similar, simple
    interpretation as strong limits. It would be interesting to
    understand the connections between the above theorem and the
    representation theory of groupoids \cite{OrloffHuef12, ConnesBook,
      Skandalis, IonescuWilliams09_2, RenaultJOT87}. Similar
    structures arise also in the representation theory of solvable Lie
    groups \cite{Beltitzas15}.  Moreover, several important examples
    of non-compact manifolds that arise in other problems lead to
    groupoids that are {\em locally} of the form studied in this paper
    (but possibly replacing $X$ by a general Lie group $G$, see
    \cite{Gerard, Moroianu2010, MelroseNistor98} and many other
    papers).}
\end{remark}

Let $A$ be a bounded operator. By definition,
$\lambda \notin \sigma_\ess(A)$ if, and only if, $A - \lambda$ is
Fredholm. For a self-adjoint operator $A$, this is equivalent to the
usual definition: ``$\lambda \in \sigma_{\ess}(A)$ if, and only if,
$\lambda$ is an accumulation point of $\sigma(A)$ or an isolated
eigenvalue of infinite multiplicity.'' The advantage of this second
definition is that it extends right away to unbounded, normal
operators (see, for instance, Remark \ref{rem.e.sp}).  A crucial
observation then is that $\lambda \notin \sigma_\ess(A)$ if, and only
if, the image $\hat{A-\lambda}$ of $A - \lambda$ in the quotient
$\rb(X)/\rk(X)$ is invertible, by Atkinson's theorem. So
$\sigma_\ess(A)=\sigma(\hat{A})$. On the other hand, the spectrum of a
normal operator in a product of $C^*$-algebras is equal to the closure
of the union of the spectra of its components. Thus the theorem above
gives right away the following corollary.

\begin{corollary}\label{co:infty} 
If $A\in\ca\rtimes X$ is normal, then $\sigma_\ess(A) =
\overline{\cup}_{\chi\in\delta(\ca)} \sigma (\tau_\chi(A))$.
\end{corollary}

If $A\in\rb(X)$, then the element $\hat{A}\in\rb(X)/\rk(X)$ may be
called the \emph{localization at infinity} of $A$.  If $A\in\ca\rtimes X$,
then its localization at infinity can be identified with the element
$\tau(A)=(\tau_\chi(A))_{\chi\in\delta(\ca)}$. Then the component
$\tau_\chi(A)\in \cbu(X)\rtimes X$ is called \emph{localization of $A$
  at $\chi\in\delta(\ca)$}. Thus \emph{the essential spectrum of
  $A\in\ca\rtimes X$ is the closure of the union of the spectra of all
  its localizations at infinity}, where the ``infinity'' is determined
by $\ca$.

We extend now the notion of localization at infinity and the formula
for the essential spectrum to certain unbounded self-adjoint operators
related to $\ca\rtimes X$. Recall that a self-adjoint operator $H$ on
a Hilbert space $\ch$ is \emph{affiliated} to a $C^*$-algebra
$\rc\subset \cb(\ch)$ if $(H-z)^{-1}\in\rc$ for some number $z$
outside the spectrum of $H$ \cite{DG1}. Clearly, this implies
$\varphi(H)\in\rc$ for all $\varphi\in\maC_0(\R)$. We shall make some
more comments on this notion after the next corollary.

\begin{corollary}\label{co:infty2} 
  If $H$ is a self-adjoint operator on $L^2(X)$ affiliated to
  $\ca\rtimes X$, then, for each $\chi\in\delta(\ca)$, the limit
  $\tau_\chi(H):=\slim_{x\to\chi}T_x^*HT_x$ exists and $\sigma_\ess(H)
  = \overline{\cup}_{\chi\in\delta(\ca)} \sigma (\tau_\chi(H))$.
\end{corollary}

The meaning of the limit above will be discussed below. 

\begin{remark}\label{rem.e.sp} {\rm
Corollary \ref{co:infty2} is an immediate consequence of Theorem
\ref{th:infty} if one thinks in terms of the functional calculus
associated to $H$.  Indeed, a real $\lambda$ does not belong to
$\sigma_\ess(H)$ if, and only if, there is $\varphi\in\maC_0(\R)$ with
$\varphi (\lambda)\neq0$ such that $\varphi(H)$ is compact.}
\end{remark}

For a detailed discussion of the notion of affiliation that we use in
this paper we refer to \cite[Section\ 8.1]{ABG} (or \cite[Appendix
  A]{DG1}). This notion is inspired by the quantum mechanical concept
of observable as introduced by J.\ Von Neumann in the 1930s (see
e.g.\ \cite[Section\ 3.2]{Var} for a general and precise mathematical
formulation) and later (1940s) developed in the Von Neumann algebra
setting. A notion of affiliation in the $C^*$-algebra setting has also
been introduced by S.\ Baaj and S.L.\ Woronowicz
\cite{BaajJulg,Woronowicz} \emph{but it is different from that we use
  here:} the contrary was erroneously stated in \cite[p.\ 534]{GI2},
but has been corrected in \cite[p.\ 278]{DG1}. For example, any
self-adjoint operator on a Hilbert space $\ch$ is affiliated to the
algebra of compact operators $\ck(\ch)$ in the sense of
Baaj-Woronowicz, but a self-adjoint operator is affiliated to
$\ck(\ch)$ in our sense if, and only if, it has purely discrete
spectrum.

According to our definition, a self-adjoint operator affiliated to an
``abstract'' $C^*$-algebra $\rc$ is the same thing as a real valued
observable affiliated to $\rc$, i.e.\ it is just a morphism
$\Phi:\maC_0(\R)\to\rc$. If $\rc\subset\cb(\ch)$, then a densely
defined self-adjoint operator $H$ defines an observable by
$\Phi(\varphi)=\varphi(H)$ for $\varphi\in\maC_0(\R)$, and we say that
$H$ is affiliated to $\rc$ if this observable is affiliated to
$\rc$. But there are observables affiliated to $\rc$ that are not of
this form: they are associated to self-adjoint operators $K$ acting in
closed subspaces $\ck\subset\ch$ as explained in the next remark.  See
\cite[Section\ 8.1.2]{ABG} for a precise statement and proof.

We now explain the meaning of $\slim_{x\to\chi}T_x^*HT_x$ for
an arbitrary self-adjoint operator $H$.

\begin{remark}\label{re:limit} {\rm Let $Y$ be a topological space,
    $z$ a point in $Y$, and let $\{H_y\}$ be a set of self-adjoint
    operators (possibly unbounded) on a Hilbert space $\ch$,
    parametrized by $Y \smallsetminus \{z\}$.  The example that we
    have in mind is $H_x := T_x^*HT_x$, $x \in X$, and $Y$ obtained
    from $X$ by adding some point of a compactification.  We say that
    \emph{$\slim_{y \to z} H_y$ exists} if the strong limit
    $\Phi(\varphi) := \slim_{y \to z}\varphi(H_y)$ exists for each
    function $\varphi\in \maC_0(\R)$.  It is easy to see that this is
    equivalent to the existence of $\slim_{y \to z}(H_y-\lambda)^{-1}$
    for some $\lambda\in\C\smallsetminus\R$. But we emphasize that
    \emph{this does not mean that there is a self-adjoint operator $K$
      on $\ch$ such that $\Phi(\varphi)=\varphi(K)$ for all
      $\varphi\in \maC_0(\R)$} if the notion of self-adjointness is
    interpreted in the usual sense, which requires the domain to be
    dense in $\ch$. However, the following is true: there is a closed
    subspace $\ck\subset\ch$ and a self-adjoint operator (in the usual
    sense) $K$ in $\ck$ such that
    $\Phi(\varphi)\Pi_\ck=\varphi(K)\Pi_\ck$ and
    $\Phi(\varphi)\Pi_\ck^\perp=0$, where $\Pi_\ck$ is the projection
    onto $\ck$.  The couple $(\ck,K)$ is uniquely defined and we write
    $\slim_{y}H_y=K$.  One may have $\ck=\{0\}$, in which case we
    write $\slim_{y}H_y=\infty$. See the Remark \ref{re:pathology} for
    an example.}
\end{remark}

\section{Spherical compactification}\label{s:sphcomp}
\protect\setcounter{equation}{0}

As before, $X$ is a finite dimensional real vector space. 
We now briefly discuss the definition of the spherical compactification 
$\oX$ of $X$, its topology, and the definition of continuous functions on $\oX$.
Recall that the sphere at infinity $\SS_X$ of $X$ is the set of all half-lines 
$\alpha= \hat a := \RR_+a$, with $\RR_+=\, (0,\infty )$ and $a \in X \smallsetminus
\{0\}$, equipped with the following topology: the open sets in $\SS_X$
are the sets of the form $\{\hat a , \, a\in O\}$ with $O$ open in $X
\smallsetminus \{0\}$.  Let us denote by $\oX$ the disjoint union $X
\cup \SS_{X}$.  If $| \, \cdot \, |$ is an arbitrary norm on $X$, then
$\SS_X$ is homeomorphic to the unit sphere $S_X := \{|\xi|=1\}$ in $X$
and $\overline{X} = X \cup \SS_X$ can be endowed with a natural
topology that makes it homeomorphic to the closed unit ball in
$X$. The resulting topological space $\oX$ will be referred to as {\em
  the spherical compactification of $X$} and is discussed in detail in
this subsection, since we need a good understanding of the continuous
functions on $\oX$.

It is convenient to have an explicit description of the topology of
$\oX$ independent of the choice of a norm.  A {\em cone} $C$ (in $X$)
is a subset of $X$ stable under the action of $\RR_+$ by
multiplication. Put differently, $C$ is a union of half-lines. A
\emph{truncated cone} (in $X$) is the intersection of a cone with the
complement of a bounded set. A half-line $\alpha$ \emph{is eventually
  in the truncated cone $C$} if there is $a\in\alpha$ such $\lambda a
\in C$ if $\lambda\geq1$.  Let $C^\dag\subset\SS_X$ be the set of
half-lines that are eventually in $C$. Then the sets of the form
$C^\dag$, with $C$ an open truncated cone, form a base of the topology
of $\SS_X$.  For any open truncated cone $C$ in $X$, we denote
$C^\ddag := C \cup C^\dag$.  Then the open sets of $X$ and the sets of
the form $C^\ddag$ form a base of the topology of $\oX$.  It is easy
to see that $\oX$ is a compact topological space in which $X$ is
densely and homeomorphically embedded. Moreover, $\oX$ induces on
$\SS_X$ the (compact) topology we defined before.

By definition, a neighborhood of $\alpha\in\SS_X$ in $\oX$ is a set
that contains a subset of the form $C^\ddag$.  We denote by
$\tilde\alpha$ the set of traces on $X$ of the neighborhoods of
$\alpha$ in $\oX$. Thus, a set belongs to $\tilde\alpha$ if, and only
if, it contains an open truncated cone that eventually contains
$\alpha$. Let $Y$ be a topological space and let $u:X\to Y$.  If
$\alpha\in\SS_X$ and $y\in Y$, then the limit $\lim_{x\to\alpha}u(x)$
(or $\lim_\alpha u$) exists and is equal to $y$ if, and only if, for
each neighborhood $V$ of $y$, there is a truncated cone $C$ that
eventually contains $\alpha$ such that $u(x)\in V$ if $x\in C$.  We
shall need the following simple lemma.

\begin{lemma}\label{lm:sx}
  Let $u:X\to\C$ be such that the limit $U(\alpha) :=
  \lim_{x\to\alpha}u(x)$ exists for each $\alpha\in\SS_X$. Then $U$ is
  a continuous function on $\SS_X$. If $u$ is continuous on $X$, then
  its extension by $U$ on $\SS_X$ is continuous on $\oX$.
\end{lemma}

\begin{proof} 
  Let us notice first that
  $\lim_{\lambda\to\infty}u(\lambda a) = U(\alpha)$ for each
  $a\in\alpha\in\SS_X$. Fix $\alpha\in\SS_X$ and $\varepsilon>0$.
  There is an open truncated cone $C$ with $\alpha\in C^\dag$ such
  that $|u(x)-U(\alpha)|<\varepsilon$ for all $x\in C$.  If
  $\beta\in C^\dag$, then, for each $b\in\beta$, we have
  $\lim_{\lambda\to\infty}u(\lambda b) = U(\beta)$. Since
  $\lambda b\in C$ for large $\lambda$, we get that
  $|U(\beta)-U(\alpha)|<\varepsilon$.  Since the sets $C^\dag$ form a
  basis of the topology of $\SS_X$, we see that $U$ is continuous.

  To prove the last statement and thus to complete the proof, let us
  extend $u$ to $\oX$ to be equal to $U$ on $\SS_X$. Then the argument used in the first
  half of the proof implies that $|u(x) - u(\alpha)| < \varepsilon$
  for all $x\in C^\ddag$. Since the sets of the form $C\sp\ddag$ form
  a basis for the system of neighborhoods of $\alpha \in \SS_X$ in
  $\SS_X$ and $\alpha$ is arbitrary, the extension of $u$ to $\oX$ by
  $U$ is continuous on $\SS_X$. Hence if $u$ is continuous on $X$,
  then its extension to $\oX$ is continuous everywhere.
\end{proof}

Since $X$ is a dense subset of $\oX$, we may identify the algebra
$\maCX$ of continuous functions on $\oX$ with a subalgebra of
$\maC(X)$.  We now give several descriptions of this subalgebra that
are independent of the preceding construction of $\oX$. Denote by
$\maC_\rmh(X)$ the subalgebra of $\cbu(X)$ consisting of functions
homogeneous of degree zero outside a compact set:
\begin{equation}\label{eq:ch}
  \maC_\rmh (X) \, := \, \{u\in C(X) ,\, \exists K\subset X \text{
    compact with } u(\lambda x)=u(x) \text{ if } x\nin K, \lambda\geq
  1 \}.
\end{equation}

\begin{lemma}\label{lm:ch}
  The algebra $\maCX$ coincides with the closure of $\maC_\rmh(X)$ in
  $\maC_\rmb(X)$. A function $u\in \maC(X)$ belongs to $\maCX$ if, for
  any compact $A\subset X\setminus\{0\}$, the limit
  $\lim_{\lambda\to+\infty} u(\lambda a)$ exists uniformly in $a\in A$
  and, in this case, for any compact $B\subset X$, we have
\begin{equation}\label{eq:cx3}
\lim_{\lambda\to+\infty} u(\lambda a + b) = u(\hat a) \quad
\text{uniformly in }\, a\in A \ \text{ and }\ b\in B \, . 
\end{equation}
Moreover: 
\begin{equation}
\maCX         = \{ u\in \maC(X) , \, \lim_{x\to\alpha} u(x) \, 
          \text{ exists for each }\, \alpha\in\SS_X \, \} \,. \label{eq:cx2}
\end{equation}
\end{lemma}

The proof is an exercise.  Observe that the topology we introduced on
$\oX$ could be introduced directly in terms of $\maC_\rmh(X)$: for
example, $\tilde\alpha$ is the filter on $X$ defined by the sets
$\{x\in X , \, |u(x)-u(\alpha)|<1\}$ when $u$ runs over
$\maC_\rmh(X)$.

The space $\maC_\rmh (X)$ is not stable under translations if the dimension of
$X$ is larger than one. However, Equation \eqref{eq:cx3} -- or a direct
argument -- immediately gives that $\maCX$ is invariant under
translations, and hence we may consider its crossed product
$\rs(X) := \maCX\rtimes X$ by the action of $X$. This crossed product
is the {\em spherical algebra of $X$} and we shall study it in the next
section. Before doing that, however, let us describe an abelian $C^*$-algebra of the same
nature, but larger than $\maCX$, which is naturally involved in the
construction of self-adjoint operators affiliated to $\re(X)$. In more
technical terms, this new algebra is the Gagliardo completion of
$\maCX$ with respect to $B(\ch^s,\ch)$ for some (hence for all) $s>0$,
where $\ch^s$ is the Sobolev space of order $s$ on $X$, $\ch=L^2(X)$,
and we embed $\maCX\subset B(\ch)\subset B(\ch^s,\ch)$ by identifying
a function with the corresponding multiplication operator. One may
find in \cite[Sec. 2.1]{ABG} a discussion of the Gagliardo completion
in a general setting, but this is not necessary for what follows.

The following notion of \emph{convergence in the mean} will be useful.
Let $\Lambda$ be a compact neighborhood of the origin in $X$ and
$\alpha\in\SS_X$. If $u\in L^1_{\mathrm{loc}}(X)$ and $c\in\C$ then
$\mlim_\alpha u=c$, or $\mlim_{x\to\alpha} u(x)=c$, means
$\lim_{a\to\alpha}\int_{a+\Lambda}|u(x)-c|\dd x=0$. Then we shall have
$\lim_{a\to\alpha}\int_{a+K}|u(x)-c|\dd x=0$ $\forall K\subset X$
compact: indeed, $K$ can be covered by a finite number of translates
of $\Lambda$ and the filter $\tilde\alpha$ is translation invariant
and coarse (see page \pageref{p:trinv}).

Obviously $\cb_0(X)=\{u\in L^\infty(X) \mid \mlim_\alpha u=0\}$ is a
closed self-adjoint ideal of $L^\infty(X)$. We also have
$\cbu(X) \cap \cb_0(X) = \maC_0(X)$ as a consequence of Lemma
\ref{lm:help} that will be proved later on for a general class of
filters.
%
%
The algebra of interest for us is:
\begin{equation}\label{eq:bg}
\cb(\oX)=\maCX+\cb_0(X) \,.
\end{equation}
This is a $C^*$-algebra because the sum of a $C^*$-subalgebra and a
closed self-adjoint ideal is always a $C^*$-algebra.  We have the
following alternative description of $\cb(\oX)$.

\begin{lemma}\label{lm:bg} The set
  $\cb(\oX)$ consists of the functions $u\in L^\infty(X)$ that have
  the following property: for any $\alpha\in\SS_X$, there is $c\in\C$
  such that $\mlim_\alpha u=c$.
\end{lemma}

\begin{proof}
It is clear that the functions in $\cb(\oX)$ have the required
property, so it suffices to prove that a function $u$ as in the
statement of the lemma may be written as a sum $u=v+w$ with
$v\in\maCX$ and $w\in\cb_0(X)$. Observe that the number $c$ is
uniquely defined by $\alpha$, and hence we may define a function
$V:\SS_X\to\C$ by the condition $V(\alpha)=c$. Thus we have  
\begin{equation}\label{eq:bg1}
\lim_{a\to\alpha}\int_{a+\Lambda}|u(x)-V(\alpha)|\dd x 
\, = \, 0 \,,\quad \forall\alpha\in\SS_X \,. 
\end{equation}
Note that \eqref{eq:bg1} means that, for any $\varepsilon>0$, there is an
open truncated cone $C$ that eventually contains $\alpha$ such that
$\int_{a+\Lambda}|u(x)-V(\alpha)|\dd x<\veps$ if $a\in C$. In
particular, if we fix $a\in\alpha$, then we get
$\lim_{r\to+\infty}\int_{ra+\Lambda}|u(x)-V(\alpha)|\dd x =0$. Let us
show now, as in the proof of Lemma \ref{lm:sx}, that $V$ is a continuous
function. Let us fix $\alpha$, $\varepsilon$ and $C$ and consider some
$\beta\in C^\dag$. By what we just proved, we have
$\lim_{r\to+\infty}\int_{rb+\Lambda}|u(x)-V(\beta)|\dd x =0$ for an
arbitrary $b\in\beta$. On the other hand, since $C$ is a truncated
open cone and $\beta$ is eventually in $C$, we have $rb\in C$ for all
large enough $r$, and hence $\int_{rb+\Lambda}|u(x)-V(\alpha)|\dd
x<\veps$. Then 
\[
|V(\alpha)-V(\beta)||\Lambda|\leq 
\int_{rb+\Lambda}|V(\alpha)-u(x)|\dd x +
\int_{rb+\Lambda}|u(x)-V(\beta)|\dd x <2\veps
\]
for large $r$, where $|\Lambda|$ is the measure of $\Lambda$. Since
these $C^\dag$ are a basis of the neighborhoods of $\alpha$ in
$\SS_X$, this proves the continuity of $V$ at the point $\alpha$. 

Now let $\theta:X\to\R$ be a continuous function such that
$\theta(x)=0$ on a neighborhood of zero and $\theta(x)=1$ for large
$x$. Then the function defined by $v(x)=\theta(x)V(\hat{x})$ for
$x\neq0$ and $v(0)=0$ belongs to $\maC_\rmh (X)$, and hence it 
belongs to $\maCX$ as well. On
the other hand, $w:=u-v\in L^\infty$ and, if we set
$W(a)=\int_{a+\Lambda}|w|\dd x$, then $\lim_{a\to\alpha}W(a) =0$ for
each $\alpha\in\SS_X$. This is because
\[
W(\alpha) \leq \int_{a+\Lambda}|u(x)-V(\alpha)|\dd x +
\sup_{x\in a+\Lambda}|v(x)-V(\alpha)| |\Lambda|
\]
and the function $v$ extended by $V$ on $\SS_X$ is continuous on
$\oX$, and hence the last term above tends to zero when $a\to\alpha$. If
$\varepsilon>0$, then for each $\alpha\in\SS_X$ there is an open
truncated cone $C_\alpha$ such that $W(a)<\varepsilon$ if
$a\in C_\alpha$. Since $\{C_\alpha^\dag\}_{\alpha\in\SS_X}$ is an open
cover of the compact $\SS_X$, there is a finite set $A\subset\SS_X$
such that $\SS_X=\cup _{\alpha\in A} C_\alpha^\dag$. Finally, it is
clear that $\cup _{\alpha\in A} C_\alpha$ is a neighborhood of
infinity in $X$ on which we have $W(a)<\veps$, so
$\lim_{a\to\infty}W(a)=0$. 
\end{proof}

Now we give a description of $\cb(\oX)$ as a Gagliardo completion of
$\maCX$. Recall that $u(q)$ is the operator of multiplication by the
function $u$ and $\ch^s$ are Sobolev spaces.

\begin{proposition}\label{pr:bg} The set
  $\cb(\oX)$ consists of the functions $u\in L^\infty(X)$ with the
  following property: there is a sequence of functions $u_n\in\maCX$
  such that
\begin{equation}\label{eq:gag}
\sup_n\|u_n\|_{L^\infty}<\infty \text{\ \ and \ }
\lim_n\|u_n(q)-u(q)\|_{\ch^s\to\ch}=0 \text{\ \ for some real } s>0 \,.
\end{equation}
\end{proposition}

\begin{proof}
  We begin by noticing that the condition \eqref{eq:gag} is
  independent of $s$. In fact, if it holds for some $s$ then clearly
  it remains true if we replace $s$ by any $t\geq s$ and it will also
  hold for $0<t<s$ because if we set $T=u_n(q)-u(q)$ then we have
\[
\|T\|_{\ch^t\to\ch} \leq \|T\|^\mu_{\ch^s\to\ch} \|T\|^\nu_{\ch\to\ch}
\quad\text{with } \mu=t/s,\ \nu=1-t/s\,.
\]
It is clear that what we really have to prove is the same assertion,
but with $\cb(\oX)$ replace by $\cb_0(X)$ and $\maCX$ replaced by
$\maC_0(X)$. Assume first that $u\in L^\infty$ can be approximated
with functions $u_n\in\maC_0$ as in \eqref{eq:gag} and let
$\varepsilon_n=\|u_n(q)-u(q)\|_{\ch^s\to\ch}$. Let
$\eta\in \cc^\infty_\rmc$ such that $\eta(x)$ is a constant $c\neq0$
on $\Lambda$ and $\|\eta\|_{\ch^s}=1$ and let us denote
$\eta_a(x)=\eta(x-a)$. Then
$\|(u-u_n)\eta_a\|_\ch\leq\varepsilon_n\|\eta_a\|_{\ch^s}=\varepsilon_n$,
and hence $\|u\eta_a\|_\ch\leq \varepsilon_n+\|u_n\eta_a\|_\ch$. Since
$u_n\in\cc_0$, there is a neighborhood $U_n$ of infinity in $X$ such
that $\|u_n\eta_a\|_\ch\leq\varepsilon_n$ if $a\in U_n$, and then we
get $\int_{a+\Lambda}|u|^2\dd x\leq 4c^{-2}\varepsilon_n^2$ for
$a\in U_n$. This clearly implies $u\in\cb_0$.

Reciprocally, let $u\in\cb_0$. Choose a positive function
$\theta\in\cc^\infty_\rmc (X)$ with $\int\theta(x)\dd x=1$ and let
$\theta_\varepsilon(x)=\theta(x/\varepsilon)/\varepsilon^{d}$ if the
dimension of $X$ is $d$.  Then it is clear that the convolution
product $=u*\theta_\varepsilon$ belongs to $\cc_0(X)$ and
$\|u*\theta_\varepsilon\|_{L^\infty}\leq\|u\|_{L^\infty}$. Hence it
suffices to prove that there is $s>0$ such that
$\|u*\theta_\varepsilon(q)-u(q)\|_{\ch^s\to\ch}\to0$ if
$\varepsilon\to0$. But this is easy because, for $s>d/2$, we have an
estimate 
\[
\|f(q)\|^2_{\ch^s\to\ch} \leq C \sup_a\int_{a+\Lambda} |f|^2\dd x
\leq C \|f\|_{L^\infty} \sup_a\int_{a+\Lambda} |f|\dd x \,.
\]
We take here $f=u*\theta_\varepsilon-u$ and note that $\int_K
|u*\theta_\varepsilon-u| \dd x \to0$ for any compact $K$ while, for
large $a$, we use the relation $\lim_{a\to\infty}\int_{a+\Lambda}|u|\dd
x=0$.  
\end{proof}

\section{The spherical algebra }\label{s:sphalg}
\protect\setcounter{equation}{0}

We study now the spherical algebra $\rs(X) := \maCX\rtimes X$ defined
in the Introduction. We begin with a lemma that will be needed in the
proof of Theorem \ref{th:sphalg}. In order to clarify the statement of
the following lemma and in order to prepare the ground for the use of
filters in other proofs, we recall now some facts about filters
\cite{Bourbaki}.
  
A {\em filter on $X$} is a set $\xi$ of subsets of $X$ such that: (1)
$X\in\xi$, (2) $\emptyset\nin\xi$, (3) if $\xi\ni F \subset G$, then
$G\in\xi$, and (4) if $F,G \in \xi$ then $F \cap G \in \xi$.  If $Y$
is a topological space and $u:X\to Y$, then $\lim_{\xi} u=y$, or
$\lim_{x\to \xi}u(x)=y$, means that $u^{-1}(V)\in\xi$ for any
neighborhood $V$ of $y$.  The filter $\xi$ on $X$ is called
\emph{translation invariant} \label{p:trinv} if, for each $F\in\xi$
and $x\in X$, we have $x+F\in\xi$. We say that $\xi$ is \emph{coarse}
if, for each $F\in\xi$ and each compact $K$ in $X$, there is $G\in\xi$
such that $G+K\subset F$.  Recall that we have denoted by
$\tilde\alpha$ the set of traces on $X$ of the neighborhoods of
$\alpha$ in $\oX$.  Clearly $\tilde\alpha$ is a translation invariant
and coarse filter on $X$ for each $\alpha\in\SS_X$.

\begin{lemma}\label{lm:help}
  Let $\xi$ be a translation invariant filter in $X$, let $\Lambda$ be a
  compact neighborhood of the origin, and $u\in\cbu(X)$. Then
\begin{equation}\label{eq:limit}
  \lim_\xi u \, = \, 0 \ \Leftrightarrow \ \lim_{a\to\xi} \,
  \int_{a+\Lambda}|u(x)| \, \d x \, = \, 0 \ \Leftrightarrow
  \ \slim_{a\to\xi}\, u(q+a)\, = \, 0 \, .
\end{equation} 
\end{lemma}

\begin{proof} 
  Recall that $u(q)$ denotes the operator of multiplication by $u$ and
  $u(q+a)$ is its translation by $a$.  We have
  $\slim_{a \to \xi} u(q+a)=0$ if, and only if,
  $\int |u(x+a)f(x)|^2\d x \to 0$ as $a\to\xi$ for all $f\in L^2(X)$,
  by the definition of the strong limit. By taking $f$ to be the
  characteristic function of the compact set $\Lambda$ and by using
  the Cauchy-Schwartz inequality, we obtain
  $\lim_{a\to\xi}\int_{a+\Lambda}|u(x)|\d x =0$.  Reciprocally, if
  this relation is satisfied then it is also satisfied with $\Lambda$
  replaced by any of its translates because $\xi$ is translation
  invariant. By summing a finite number of such relations, we get
  $\lim_{a\to\xi}\int_{a+K}|u(x)|\d x =0$ for any compact $K$. Since
  $u$ is bounded, we also obtain
  $\lim_{a\to\xi}\int_{a+K}|u(x)|^2\d x =0$ 
  and so $\lim_{a\to\xi}\int |u(x+a)f(x)|^2\d x =0$, for any simple
  function $f$. Using again the boundedness of $u$, we then obtain
  $\lim_{a\to\xi}\int |u(x+a)f(x)|^2\d x =0$ for $f\in L^2(X)$.

We now show that $\lim_\xi u=0$ is equivalent to
$\lim_{a\to\xi}\int_{a+\Lambda}|u(x)|\d x=0$. We may assume $u\geq0$, and
since $u$ and $a\mapsto \int_{a+\Lambda}u(x)\d x$ are bounded uniformly
continuous functions, we may also assume that $\xi$ is
coarse\footnote{This follows from \cite[Lemma 2.2]{GI3} and a simple
  argument, which shows that the round envelope of a translation
  invariant filter is coarse. We do not include the details since in
  our applications $\xi=\tilde\alpha$, which is coarse.}.  If
$\lim_\xi u=0$, then $\{u<\varepsilon\}\in\xi$, for any
$\varepsilon>0$. Since $\xi$ is coarse, there is $F\in\xi$ such that
$F+\Lambda\subset\{u<\varepsilon\}$, and hence, if $a\in F$, then
$\int_{a+\Lambda}u (x)\d x\leq \varepsilon |a+\Lambda|= \varepsilon |\Lambda|$.  Thus we
have $\lim_{a\to\xi}\int_{a+\Lambda} u(x)\d x =0$.  Conversely, assume that
this last condition is satisfied and let $\varepsilon>0$ Since $u$ is
uniformly continuous, there is a compact symmetric neighborhood
$L\subset \Lambda$ of zero such that $|u(x)-u(y)|<\varepsilon$ if
$x,y\in L$. Then
\[
   u(a) |L| \, = \, \int_{a+L} (u(a)-u(x))\d x + \int_{a+L} u(x)\d x \leq
   \varepsilon |L| + \int_{a+L} u(x)\d x \,,
\]
and hence $\limsup_{a\to\xi} u(a)\leq \varepsilon |L|$. 
\end{proof}

Recall that $\tau_a(S)=T_a^* S T_a$, where the unitary translation
operators $T_a$ are defined in \eqref{eq:trans.mult}.

\begin{theorem}\label{th:sphalg} 
The algebra $\rs(X) := \maCX\rtimes X$ consists of the $S \in \rb(X)$
that have the position-momentum limit property and are such that
$\slim_{a\to\alpha} \tau_a(S)^{(*)}$ exists $\forall\ \alpha\in\SS_X$.
\end{theorem}

\begin{proof} 
Let $\ra$ be the set of bounded operators that have the properties in
the statement of the theorem. We first show that $\rs(X) \subset \ra$.
Recall that in the concrete realization we mentioned above,
$\cbu(X)\rtimes X$ is identified with the norm closed linear space
generated by the operators $S=u(q)v(p)$ with $u\in \cbu(X)$ and
$v\in\maC_0(X^*)$, while $\maCX\rtimes X$ is the norm closed subspace
generated by the same type of operators, but with $u\in \maCX$.  It
follows that an operator $S = u(q)v(p)$, with $u\in \maCX$, has the
position-momentum limit property and
\begin{equation}\label{eq:liminfty}
  \slim_{a\to\alpha} T_a^* S T_a \, = \, \slim_{a \to \alpha} u(q+a)v(p)
  =u(\alpha)v(p) \,,
\end{equation}
because of relation \eqref{eq:cx3}. Thus $\rs(X)\subset \ra$ and it
remains to prove the opposite inclusion.

It is clear that $\ra$ is a $C^*$-algebra. From \cite[Theorem
  3.7]{GI3} it follows that $\ra$ is a crossed product $\ra=\ca\rtimes
X$ with $\ca\subset\cbu(X)$ if, and only if, $\ra\subset
\cbu(X)\rtimes X$ and
\begin{equation}\label{eq:cond}
  x\in X, k\in X^*, S\in\ra \Rightarrow T_x S\in\ra \text{ and }
  M_k S M_k^*\in\ra .
\end{equation}
By the definition of $\ra$, the condition $\ra\subset \cbu(X)\rtimes
X$ is obviously satisfied. Moreover, we have $T_a^*T_xS
T_a = T_x T_a^* S T_a$ and $T_a^* M_k S M_k^* T_a = M_k T_a^*S T_a M_k^*$, and
hence the last two conditions in \eqref{eq:cond} are also satisfied.
Therefore $\ra$ is a crossed product. Theorem 3.7 form \cite{GI3}
gives more: the unique translation invariant $C^*$-subalgebra
$\ca\subset \cbu(X)$ such that $\ra=\ca\rtimes X$ is the set of
$u \in\cbu(X)$ such that $u(q)v(p)$ and $\bar{u}(q)v(p)$ belong to
$\ra$ if $v \in \maC_0(X^*)$.  In our case, we see that $\ca$ is the
set of all $u \in\cbu(X)$ such that $\slim_{a\to\alpha} T_a^*u(q)^{(*)}
T_a v(p)$ exists for all $\alpha\in\SS_X$ and $v\in\maC_0(X^*)$.  But
the operators $T_a^*u(q)^{(*)}T_a = u^{(*)}(q+a)$ are normal and
uniformly bounded and the union of the ranges of the operators $v(p)$
is dense in $L^2(X)$, and hence
\[
  \ca \, = \, \{ u\in\cbu(X) , \, \exists\, \slim_{a\to\alpha}u(q+a)
  \ \forall\alpha\in \SS_X \}.
\]
Let us fix $\alpha$ and let $u\in\cbu(X)$ be such that the limit
$\slim_{a\to\alpha} u(q+a)$ exists. This limit is a function, but
since the filter $\tilde\alpha$ is translation invariant, this
function must be in fact a constant $c$. Applying Lemma \ref{lm:help}
to $u-c$ we get $\lim_\alpha u = c$. Lemma \ref{lm:ch} then gives
\[
  \ca \, = \, \{ u\in\cbu(X) , \, \exists\, \lim_{x\to\alpha}u(x)
  \ \forall\alpha\in \SS_X \} \, = \, \maCX .
\]   
This proves the theorem.
\end{proof}

For each $\alpha\in\SS_X$ and $S\in \rs(X) := \maCX\rtimes X$, we then
define
\begin{equation}\label{eq.def.Pa}
  \tau_\alpha(S) \, := \, \slim_{a\to\alpha} T_a^* S T_a \,.
\end{equation}

\begin{theorem}\label{th:ker}
  If $S\in \rs(X)$ and $\alpha\in\SS_X$, then $\tau_\alpha(S) \in
  C^*(X)$ and the map $\tau(S):\alpha\mapsto\tau_\alpha(S)$ is norm
  continuous, and hence $\tau: \rs(X) \to C(\SS_X)\otimes C^*(X)$. The
  resulting morphism $\tau$ is a surjective morphism and its kernel is
  the set $\rk(X)=\maC_0(X)\rtimes X$ of compact operators on
  $L^2(X)$.  Hence we have a natural identification
\begin{equation}\label{eq:quot}
  \rs(X) /\rk(X) \, \cong \, \cc(\SS_X)\otimes C^*(X) 
   \, \cong \, \cc_0(\SS_X \times X\sp{*}) \, . 
\end{equation}
\end{theorem}

\begin{proof} If $S=u(q)v(p)$, then, from \eqref{eq:liminfty}, we get
  $\tau_\alpha(u(q)v(p))=u(\alpha)v(p)$, and thus $\tau(S)=\tilde
  u\otimes v(p)$, where $\tilde u$ is the restriction of $u:\oX\to\C$
  to $\SS_X$. The first assertion of the theorem then follows from the
  density in $\rs(X)$ of the linear space generated by the operators
  of the form $u(q)v(p)$.  The fact that $\tau_\alpha$ are morphisms
  follows from their definition as strong limits, and it implies the
  fact that $\tau$ is a morphism.  Since the range of a morphism is
  closed and $u\mapsto\tilde u$ is a surjective map $\maCX \to
  C(\SS_X)$, we get the surjectivity of $\tau$.  It remains to show
  that $\ker\tau=\maC_0(X)\rtimes X$.  By what we have proved,
  $\ra_0=\ker\tau$ is the set of operators $S$ that have the
  position-momentum property and are such that $\slim_{a\to\alpha}
  T_aS T_a=0$ for all $\alpha\in\SS_X$.  The argument of the proof of
  Theorem \ref{th:sphalg} with $\ra$ replaced by $\ra_0$ shows that
  $\ra_0=\ca_0\rtimes X$, with $\ca_0$ equal to the set of all
  $u\in\cbu(X)$ such that $\lim_{x\to\alpha}u(x)=0$ for all $\alpha\in
  \SS_X$. Therefore $\ca_0 =\maC_0(X)$.
\end{proof}

\begin{remark}\label{re:alpha} {\rm The fact that $\tau_\alpha(S)$
    belongs to $C^*(X)$ can be understood more generally as follows.
    Since the filter $\tilde\alpha$ is translation invariant, if $S$
    is an arbitrary bounded operator such that the limit
    $S_\alpha := \slim_{a\to\alpha} T_a^* S T_a$ exists, then
    $S_\alpha$ commutes with all the $T_x$, and hence $S$ is of the
    form $v(p)$, for some $v\in L^\infty(X^*)$. If $S$ has the
    position-momentum limit property, then it is clear that $S_\alpha$
    also has the position-momentum limit property, which forces
    $v\in\maC_0(X^*)$.  }
\end{remark}

We have the following consequences of the above theorem. We notice that we do not
need closure in the union in the following results since $\tau_\alpha(S)$ depends \emph{norm}
continuously on $\alpha$. See \cite{NistorPrudhon} for a general
discussion of the need of closures of the unions in results of this
type.

\begin{corollary}\label{co:ker0}
 Let $S \in \rs(X)$ be a normal element. Then
 $\sigma_\ess(S)=\cup_\alpha \sigma(\tau_\alpha(S))$.
\end{corollary}

Similarly, we have the following.

\begin{corollary}\label{co:ker}
  Let $H$ be a self-adjoint operator affiliated to $\rs(X)$. Then for
  each $\alpha\in\SS_X$ the limit $\alpha.H:=\slim_{a\to\alpha}
  T_a^*HT_a $ exists and $\sigma_\ess(H)=\cup_\alpha
  \sigma(\alpha.H)$.
\end{corollary}

For the proof and for the meaning of the limit above, see Remark
\ref{re:limit}.

We give now the simplest concrete application of Corollary
\ref{co:ker}.  The boundedness condition on $V$ can be eliminated, but
this requires some technicalities, which will be discussed in the next
section.

\begin{proposition}\label{pr:ess1}
  Let $H=h(p)+V$, where $h:X^*\to\RR$ is a continuous proper
  function and $V$ is a bounded symmetric linear operator on $L^2(X)$
  satisfying
\begin{enumerate}[(i)]
\item $ \lim_{k\to0}\| [M_k, V]   \|=0$,
\item $\alpha.V:=\slim_{a\to\alpha} T_a^* V T_a$ exists for each
  $\alpha\in\SS_X$.
\end{enumerate}
Then $H$ is affiliated to $\rs(X)$, we have $\alpha.H=h(p)+\alpha.V$,
and $\sigma_\ess(H)=\cup_\alpha \sigma(\alpha.H)$. Moreover, for each
$\alpha\in\SS_X$, there is a function $v_\alpha\in\cbu(X^*)$ such that
$\alpha.V=v_\alpha(p)$.
\end{proposition} 

\begin{proof} First we have to check that the self-adjoint operator $H$ is
affiliated to $\rs(X)$.  For this, it suffices to prove that there is
a number $z$ such that the operator $S=(H-z)^{-1}$ satisfies the
conditions of Theorem \ref{th:sphalg}. To check the position-momentum
limit property we have to prove that $(T_x-1)S$ and $[M_k,S] $ tend to
zero in norm when $x\to0$ and $k\to0$ (the condition involving $S^*$
will then also be satisfied since $S\sp{*}$ is of the same form as
$S$).  Since the range of $S$ is the domain of $h(p)$, the first
condition is clearly satisfied. If we denote $S_0 = (h(p) - z)\sp{-1}$
and choose $z$ such that $\|VS_0\|<1$, then we have
$S=S_0(1+VS_0)^{-1}$ and $S_0\in C^*(X)$, and hence $[M_k,S_0]$ tends to
zero in norm as $k\to0$. It remains to be shown that $(1+VS_0)^{-1}$
also satisfies this condition: but this is clear because the set of
bounded operators $A$ such that $\| [M_k,A] \|\to0$ is a
$C^*$-algebra, and hence a full subalgebra of $\rb(X)$.

The fact that $\slim_{a\to\alpha}T_a^* S T_a$ exists and is equal to
$\alpha.S=(\alpha.H-z)^{-1}$ for each $\alpha\in\SS_X$ is an easy
consequence of the relation $T_a^* S T_a=S_0(1+T_a^*VT_a S_0)^{-1}$.

Finally, to show that $\alpha.V=v_\alpha(p)$, for some
$v_\alpha\in\cbu(X^*)$, we use the argument of Remark
\ref{re:alpha}. Indeed, we shall have this representation for some
bounded Borel function $v_\alpha$, which must be uniformly continuous
because $ \lim_{k\to0}\| [M_k, \alpha.V] \| = 0$.
\end{proof}

\begin{example}\label{ex:1}{\rm A typical example is when $V$ is the
    operator of multiplication by a bounded Borel function $V:X\to\R$
    such that $V(\alpha):=\lim_{x\to\alpha}V(x)$ exists for each
    $\alpha\in \SS_X$. Then $\alpha.V$ is the operator of
    multiplication by the number $V(\alpha)$. Note that by Lemma
    \ref{lm:sx} the limit function $\alpha\mapsto V(\alpha)$ is
    continuous on $\SS_X$, even if $V$ is not continuous on $X$.  }
\end{example}

\begin{remark}\label{rem.SG}{\rm The constructions in this section are
    related to the ones involving the so called ``SG-calculus'' or
    ``scattering calculus,'' see \cite{dasguptaWongSG08,
      dasguptaWongGR10, HMV2004, nicolaRodinoSG03, parentiSG72,
      rabinovichRochGe08} and the references therein. In fact, the
    closure in norm of the algebra of order $-1$,
    $SG$-pseudodifferential operators coincides with the algebra
    $\maC(\overline{X})\rtimes X$.}
\end{remark}

\section{Affiliation criteria} \label{s:affS}

We now recall, for the benefit of the reader, a little bit of the
formalism that we shall use below. If $H$ is a self-adjoint operator
on a Hilbert space $\ch$, then the domain of $|H|^{1/2}$ equipped with
the graph topology is called the \emph{form domain} of $H$. If we
denote it $\cg$, then we have a natural continuous embedding
$\cg\subset\ch\subset \cg^*$, where $\cg^*$ is the space adjoint of
$\cg$ (space of conjugate linear continuous forms on $\cg$).  The
operator $H:D(H)\to\ch$ extends to a continuous symmetric operator
$\hat{H}\in \cb(\cg,\cg^*)$, which has the following property: a
complex number $z$ belongs to the resolvent set of $H$ if, and only
if, $\hat{H}-z$ is a bijective map $\cg\to\cg^*$. In this case,
$(H-z)^{-1}$ coincides with the restriction of $(\hat{H}-z)^{-1}$ to
$\ch$. Conversely, let $\cg$ be a Hilbert space densely and
continuously embedded in $\ch$.  If $L:\cg\to\cg^*$ is a symmetric
operator, then the \emph{operator induced by $L$ in $\ch$} is the
operator $H$ in $\ch$ whose domain is the set of $u\in\cg$ such that
$Lu\in\ch$ given by $H=L|D(H)$. If $L-z:\cg\to\cg^*$ is a bijective
map for some complex $z$, then $D(H)$ is a dense subspace of $\ch$,
the operator $H$ is self-adjoint, and $\hat{H}=L$. If $L$ is bounded
from below, then $\cg$ coincides with the form domain of $H$. From now
on, we shall drop the ``hat'' from the notation $\hat{H}$ and
write simply $H$ for the extended operator when there is no danger of
confusion.

\begin{lemma}\label{lm:easy}
  Let $\cg$ be a Hilbert space densely and continuously embedded in
  $L^2(X)$. Then the following conditions are equivalent:
\begin{itemize}
\item The operators $T_x$ and $M_k$ leave invariant $\cg$, we have
  $\|T_x\|_{\cb(\cg)}\leq C$ for a number $C$ independent of $x$, and
  $\lim_{x\to0}\|T_x-1\|_{\cg\to\ch}=0$.
\item $\cg=D(w(p))$ for some proper Borel function
  $w:X^*\to [1,\infty )$ that has the following property: there
  exists a compact neighborhood $\Lambda$ of zero in $X^*$ and a
  number $c > 0$ such that $\sup_{\ell\in \Lambda}w(k+\ell)\leq cw(k)$
  for all $k\in X^*$.
\end{itemize}
\end{lemma}

\begin{proof} We discuss only the nontrivial implication. Denote
  $\ch=L^2(X)$. Since $\{T_x\}_{x\in X}$ is a strongly continuous
  unitary group in $\ch$ that leaves $\cg$ invariant, the restrictions
  $T_x|\cg$ form a $C_0$-group in $\cg$, which by assumption is
  (uniformly) bounded. It is well known that this implies that there
  is a Hilbert structure on $\cg$, equivalent to the initial one, for
  which the operators $T_x|\cg$ are unitary (indeed, $\R$ is
  amenable). Thus, from now on, we may assume that the operators $T_x$
  are unitary in $\cg$. Then, by the Friedrichs theorem, there\ exists
  a unique self-adjoint operator $G$ on $\ch$ with the following
  properties:
\begin{enumerate}[(i)]
\item $G\geq c>0$ for some number $c$; 
\item $\cg=D(G)$;
\item for all $g\in\cg$, we have $\|g\|_\cg=\|Gg\|$.
\end{enumerate}
By hypothesis, the unitary operator $T_x$ leaves invariant the domain
of $G$ and $\|g\|_\cg=\|GT_xg\|=\|T_x^*GT_xg\|$ for all $g\in D(G)$
and $x\in X$. By the uniqueness of $G$, we have $T_x^*GT_x=G$, and
hence $G$ commutes with all translations. It follows that there is a
Borel function $w:X^*\to[c,\infty )$ such that $G=w(p)$. We have
\begin{align*}
  \|(T_x-1)\|_{\cg\to\ch} &=\|(T_x-1)G^{-1}\|_{\ch\to\ch}
  =\|(\e^{\rmi xP}-1)w^{-1}(P)\|_{\ch\to\ch}  \\
  &=\ess\!\!\sup_{p\in X^*}|(\e^{\rmi xp}-1)w^{-1}(p)|
\end{align*}
and $w^{-1}$ is a bounded Borel function. It follows that $w^{-1}$
tends to zero at infinity. 

Now we shall use the fact that the operators $M_k$ also leave invariant
$\cg$. Then the group induced by $\{M_k\}$ in $\cg$ is of class
$C_0$. In particular, $\|w(p)M_\ell g\|\leq C\|w(p)g\|$ if
$\ell\in \Lambda$ and $g\in\cg$. Since
$M_\ell^*w(p)M_\ell = w(p+\ell)$, we get
$\|w(p+\ell)w(p)^{-1}f\|\leq C\|f\|$ for $\ell \in \Lambda$ and
$f\in\ch$, which means that $w(k+\ell)w(k)^{-1}\leq C$ for all
$k\in X^*$ and $\ell\in \Lambda$. Thus for each fixed $k$, $w$ is
bounded on $k+\Lambda$, and hence $w$ is bounded on any compact.
\end{proof}

The next result is a general criterion of affiliation to $\rs(X)$ for
semi-bounded operators.

\begin{theorem}\label{th:hconj}
  Let $H$ be a self-adjoint operator on $L^2(X)$ that is bounded from
  below and its form domain $\cg$ satisfies the conditions of Lemma
  \ref{lm:easy}. Assume that $\|[M_k, H]\|_{\cg\to\cg^*}\to 0$ as
  ${k\to0}$ and that the limit $\alpha.H:=\lim_{a\to\alpha} T_a^* H
  T_a$ exists strongly in $\cb(\cg,\cg^*)$, for all $ \alpha\in\SS_X$.
  Then $H$ is affiliated to $\rs(X)$, for each $\alpha\in\SS_X$ the
  operator in $L^2(X)$ associated to $\alpha.H$ is self-adjoint, and
  $\sigma_\ess(H)=\cup_\alpha \sigma(\alpha.H)$.
\end{theorem}

\begin{proof} 
We shall use Theorem \ref{th:sphalg} and then Corollary \ref{co:ker}.
We first check that the first condition of Theorem \ref{th:sphalg} is
satisfied. Denote $R=(H+\rmi)^{-1}$.  We have
$\|(T_x-1)\|_{\cg\to\ch}=\|(T_x-1)|R|^{1/2}\|$, and hence $\lim_{x\to
  0} \|(T_x-1)R\|=0$.  As explained above, $R$ extends uniquely to an
operator $\hat{R} \in \cb(\cg^*,\cg)$. The operators $M_k$ leave $\cg$
invariant and thus extend continuously to $\cg^*$. Consequently, we
have $[M_k, \hat{R}] = \hat{R}[H,M_k]\hat{R}$. Hence we get $
\lim_{k\to0}\| [M_k, \hat{R}] \|_{\cg^*\to\cg}=0$, which is more than
enough to show that $H$ has the position-momentum limit property.

To finish the proof of the proposition, it is enough to check the last
condition of Theorem \ref{th:sphalg} and then use Corollary
\ref{co:ker}. Clearly $\alpha.H:\cg\to\cg^*$ satisfies
$\braket{g}{\alpha.H g}=\lim_{a\to\alpha}\braket{T_ag}{H T_ag}$ for
each $g\in\cg$. Note that since we assumed $H$ bounded from below, we
may assume that $H\geq1$ (otherwise we add to it a sufficiently large
number). Then, if $w$ is as in Lemma \ref{lm:easy}, the norm
$\|w(p)g\|$ defines the topology of $\cg$, and hence
$\braket{u}{Hu}\geq c\|w(p)u\|^2$ for some number $c$ and all
$u\in\cg$. This implies $\braket{T_ag}{HT_ag}\geq c\|w(p)T_a
g\|^2=c\|w(p)g\|^2 $. Thus we get $\braket{g}{\alpha.H g}\geq
c\|w(p)g\|^2 $, and hence $\alpha.H$ is a bijective map
$\cg\to\cg^*$. Next, to simplify the notation, we set $H_a=T_a^* H
T_a, H_\alpha=\alpha.H$, and note that since these operators are
isomorphisms $\cg\to\cg^*$, we have \(
H_a^{-1}-H_\alpha^{-1}=H_a^{-1}(H_\alpha-H_a) H_\alpha^{-1} \) as
operators $\cg^*\to\cg$, which clearly implies
$\slim_{a\to\alpha}T_a^*H^{-1}T_a=H_\alpha^{-1}$ in $\cb(\cg^*,\cg)$,
which is more than enough to prove the convergence of the self-adjoint
operators $T_a^*HT_a$ to the self-adjoint operator $\alpha.H$ in
$L^2(X)$ in the sense required in Corollary \ref{co:ker}.
\end{proof}

In the next theorem, we consider operators of the form $h(p)+V$, with
$V$ unbounded, and impose on $h$ the simplest conditions that ensure
that the form domain of $h(p)$ is stable under the operators
$M_k$. Obviously, much more general conditions could have been used to
obtain the same result, however, these conditions are well adapted to
elliptic operators with non-smooth coefficients.  For any real number
$s$, let $\ch^s\equiv\ch^s(X)$ be the Sobolev space of order $s$ on
$X$.  Also, let $|\cdot|$ be any norm on $X^*$.

\begin{theorem}\label{th:ess2}
Let $h:X^*\to[0,\infty )$ be a locally Lipschitz function with
  derivative $h'$ such that, for some real numbers $c,s>0$ and all
  $k\in X^*$ with $|k|>1$, we have:
\begin{equation}\label{eq:ess2}
c^{-1}|k|^{2s}\leq h(k) \leq c |k|^{2s} \quad\text{and}\quad 
|h' (k)|\leq c |k|^{2s} \, .
\end{equation}
Let $ V : \ch^s\to\ch^{-s}$ symmetric such that
$V \geq -\gamma h(p) -\delta$, for some numbers $\gamma,\delta$, with
$\gamma<1$. We assume that $V$ satisfies the following two conditions:
\begin{enumerate}[(i)]
\item $\lim_{k\to0}\| [M_k, V]   \|_{\ch^s\to\ch^{-s}}=0$,
\item $\forall\alpha\in\SS_X$ the limit $\alpha.V :=
  \slim_{a\to\alpha} T_a^* V T_a$ exists strongly in
  $B(\ch^s,\ch^{-s})$.
\end{enumerate}
Then $h(p)+V$ and $h(p)+\alpha.V$ are symmetric operators $\ch^s \to
\ch^{-s}$ and the operators $H$ and $\alpha.H$ associated to them in
$L^2(X)$ are self-adjoint and affiliated to $\rs(X) := \cc(\oX)
\rtimes X$.  Moreover, the essential spectrum of $H$ is given by the
relation $\sigma_\ess(H) = \cup_\alpha \sigma(\alpha.H)$.
\end{theorem} 

\begin{proof}
  If we denote $w=\sqrt{1+h}$, then the form domain $\cg$ of $h(p)$ is
  $\cg=D(w(p))=\ch^s$. The second condition of Lemma \ref{lm:easy}
  will be satisfied if $\sup_{|\ell|<1}h(k+\ell)\leq c(1+h(k))$ for
  some number $c>0$, which is clearly true under our assumptions on
  $h$. Then note that we have $h(p)+V+\delta+1\geq (1-\gamma) h(p)+1$ as
  operators $\cg\to\cg^*$ and this estimate remains true if $V$ is
  replaced by $\alpha.V$. It follows that $h(p)+V+\delta+1:\cg\to\cg^*$
  is bijective, and hence the operator $H$ induced by $h(p)+V$ in $L^2(X)$
  is self-adjoint. The same method applies to $\alpha.H$. Thus the
  conditions of Theorem \ref{th:hconj} are satisfied and we may use it
  to get the results of the present theorem.
\end{proof}

\begin{example}\label{ex:2}{\rm The simplest examples that are covered
    by the preceding result are the usual elliptic symmetric operators
    $\sum_{|\mu|,|\nu|\leq m} p^\mu g_{\mu\nu}p^\nu$ with bounded
    measurable coefficients $g_{\mu\nu}$ such that
    $\lim_{a\to\alpha}g_{\mu\nu}(x+a)=g^\alpha_{\mu\nu}$ exists for
    each $x\in X$ and $\alpha\in\SS_X$. Here $X=\R^n$ and the
    notations are as in \eqref{eq:L} and we assume \eqref{eq:L1}.
    Then the localizations at infinity will be the operators
    $\alpha.H$, which are of the same form, but with the functions
    $g_{\mu\nu}$ replaced by the numbers $g_{\mu\nu}^\alpha$. Note
    that $\alpha\mapsto g_{\mu\nu}^\alpha$ are continuous
    functions. We can also allow the lower order coefficients
    $g_{\mu\nu}$ to be suitable singular functions or even suitable
    non-local operators.}
\end{example}

\begin{remark}\label{re:pathology}{\rm The situations considered in
    Example \ref{ex:2} could give the wrong impression that the
    localizations at infinity $\alpha.H$ are self-adjoint operators in
    the usual sense on $L^2(X)$. The following example shows that this
    is not true even in simple situations. Let $H=p^2+v(q)$ in
    $L^2(\R)$ with $v(x)=0$ if $x<0$ and $v(x)=x$ if $x\geq0$. It is
    clear that $H$ has the position-momentum limit property and, if
    $R=(H+1)^{-1}$, it is not difficult to check that
    $\slim_{a\to+\infty}T_a^*RT_a=0$ and
    $\slim_{a\to-\infty}T_a^*RT_a=(p^2+1)^{-1}$.  Indeed, the
    translated potentials $v_a(x)=(T_a^*v(q)T_a)(x)=v(x+a)$ form an
    increasing family, i.e. $v_a\leq v_b$ if $a\leq b$, such that
    $v_a(x)\to +\infty$ if $a\to+\infty$ and $v_a(x)\to0$ if
    $a\to-\infty$. Thus $H_{+\infty}=\infty$, in the sense that its
    domain is equal to $\{0\}$, and $H_{-\infty}=p^2$.  }\end{remark}

\begin{remark}\label{re:stark}{\rm In view of the Remark
    \ref{re:pathology}, it is tempting to see what happens in the case
    of the \emph{Stark Hamiltonian} $H=p^2+q$. In fact, the situation 
    in the case of the Stark Hamiltonian is much
    worse: \emph{$H$ has not the position-momentum limit property (both
    conditions of Definition \ref{df:pmlp} are violated by the
    resolvent of $H$)} and we have
    $\slim_{|a|\to\infty}T_a^*HT_a=\infty$ and
    $\slim_{|k|\to\infty}M_k^*HM_k=\infty$, while the essential
    spectrum of $H$ is $\R$.  So the localizations of $H$ in the
    regions $|p|\sim\infty$ and $|q|\sim\infty$ say nothing about the
    essential spectrum of $H$.  }
\end{remark}

We now recall some definitions and a result that can be used for
operators that are not semi-bounded and that will be especially useful
in the general context of $N$-body Hamiltonians.

Let $H_0$ be a self-adjoint operator on a Hilbert space $\ch$ with
form domain $\cg$. We say that a continuous sesquilinear form $V$ on
$\cg$ (i.e. a symmetric linear map $V:\cg\to\cg^*$) is a
\emph{standard form perturbation} of $H_0$ if there are positive
numbers $\gamma,\delta$ with $\gamma<1$ such that either
$\pm V\leq \gamma|H_0|+\delta$ or $H_0$ is bounded from below and
$V\geq-\gamma H_0-\delta$. In this case, the operator in $H$ in $\ch$
associated to $H_0+V:\cg\to\cg^*$ is self-adjoint (see the comments at
the beginning of this section).

We do not recall the definition of strict affiliation, but we use the
following fact: a self-adjoint operator $H$ is \emph{strictly
  affiliated} to a $C^*$-algebra $\rc$ of operators on $\ch$ if, and
only if, there is $\theta\in\maC_0(\R)$ with $\theta(0)=1$ such that
$\lim_{\varepsilon\to0}\|\theta(\varepsilon H)C-C\|=0$, for all
$C\in\rc$. The following is a consequence of Theorem 2.8 and Lemma 2.9
in \cite{DG1}.

\begin{theorem}\label{th:recall0}
  Let $H_0$ be a self-adjoint operator, $V$ a standard form
  perturbation of $H_0$, and $H=H_0+V$ the self-adjoint operator
  defined above. Assume that $H_0$ is strictly affiliated to a
  $C^*$-algebra $\rc$ of operators on $\ch$. If there is
  $\phi\in\maC_0(\R)$ with $\phi(x)\sim|x|^{-1/2}$ for large $x$ such
  that $\phi(H_0)^2V\phi(H_0)\in\rc$, then $H$ is also strictly
  affiliated to $\rc$.
\end{theorem}

We may of course replace $\phi(H_0)^2V\phi(H_0)\in\rc$ by the more
symmetric and simpler looking condition $\phi(H_0)V\phi(H_0)\in\rc$,
but this will not cover in the applications the case when the operator
$V$ is of the same order as $H_0$. For operators bounded from below
we have:

\begin{theorem}\label{th:recall}
  Let $H_0$ be a positive operator strictly affiliated to a
  $C^*$-algebra of operators $\rc$ on a Hilbert space $\ch$.  Let $V$
  be a continuous sesquilinear form on $D(H_0^{1/2})$ such that
  $V\geq -\gamma H_0-\delta$ with $\gamma<1$. If
  $\varphi(H_0)V (H_0+1)^{-1/2}\in\rc$ for any
  $\varphi\in\cc_\rmc(\R)$, then the form sum $H = H_0+V$ is a
  self-adjoint operator strictly affiliated to $\rc$.
\end{theorem}

The next proposition is an immediate consequence of Theorem
\ref{th:recall}. Note that below the form domain of $h(p)$ is the
domain of $k(p)$, where $k$ is the function $|h|^{1/2}$. It is clear
that if $h$ is a proper continuous function, then $h(p)$ is strictly
affiliated to $\rs(X)$.

\begin{proposition}\label{pr:cor}
  Let $H=h(p)+V$, where $h:X^*\to\R$ is a continuous proper function
  and $V$ is a standard form perturbation of $h(p)$. If
  $(1+|h(p)|)^{-1}V(1+|h(p)|)^{-1/2}$ belongs to $\rs(X)$, then $H$ is
  strictly affiliated to $\rs(X)$.
\end{proposition}

We may replace above $(1+|h|)^{-1/2}$ by any function of the form
$\theta\circ h$ with $\theta$ as in Theorem \ref{th:recall}. Indeed,
$\maC_\rmb(X^*)$ is obviously included in the multiplier algebra of
$\rs(X)$.

For $0\leq s\leq1$, let $\cg^s := D(|h(p)|^s)$ equipped with the graph
topology and let $\cg^{-s}$ be its adjoint space. So $\cg^1=\cg$,
$\cg^0=\ch$, and $\cg^{-1}=\cg^*$. If $V$ is a continuous symmetric
form on $\cg$ such that $V\cg^1\subset\cg^{-s}$ for some $s<1$, then
for each $\gamma>0$ there is a real $\delta$ such that
$\pm V\leq \gamma|h(p)|+\delta$, and hence $V$ is a standard form
perturbation of $h(p)$ and $H$ is well defined.

\begin{corollary}\label{co:cor}
  Let $H=h(p)+V$, where $h:X^*\to\R$ is a continuous proper function,
  and let $V$ be a continuous symmetric form on $\cg$ such that
  $V\cg^1\subset\cg^{-s}$ with $s<1$. Let $\phi$ be a smooth function
  such that $\phi(x)\sim|x|^{-1/2}$ for large $x$ and denote
  $L=\phi(H_0)V\phi(H_0)$.  If $\lim_{k\to 0}\|[M_k,L]\|=0$ and
  $\alpha.V=\slim_{a\to\alpha} T_a^*VT_a$ exists in $\cb(\cg,\cg^*)$
  for each $\alpha\in\SS_X$, then $H$ is affiliated to $\rs(X)$, we
  have $\alpha.H=h(p)+\alpha.V$, and $\sigma_\ess(H)=\cup_\alpha
  \sigma(\alpha.H)$.
\end{corollary}


\section{N-body type interactions}\label{s:nbd}
\protect\setcounter{equation}{0}
\protect\setcounter{tocdepth}{2}

In this section we introduce and study the algebra of potentials (or
elementary interactions) in the $N$-body case. We will implicitly
assume $X$ of dimension $\geq2$, because in the one dimensional case
the algebra $\re(X)$ defined in \eqref{eq:mainham} coincides with
$\rs(X)$.

\subsection{N-body framework}
\addtocontents{toc}{\protect\setcounter{tocdepth}{2}}

The framework that we introduce here allows us to define and classify
$N$-body Hamiltonians in terms of the complexity of the interactions
inside subsets of particles.

Assume that for each finite dimensional real vector space $E$ a
translation invariant $C^*$-subalgebra $\cp(E)$ of $\cbu(E)$ has been
specified (the letter $\cp$ should suggest ``potentials''). Then, for
each subspace $Y\subset X$, we get a translation invariant subalgebra
$\cp(X/Y)\subset\cbu(X)$. Let us denote by
$\langle \, A_\alpha , \alpha \in I \, \rangle$ the norm closed
subalgebra generated by a family $\{A_\alpha\}_{\alpha\in I}$ of sets
$A_\alpha \subset \cbu(X)$. Then we let
\begin{equation}\label{eq:def.RX}
\cR_{\cp}(X) \,:= \, \langle \, \cp(X/Y) ,\, Y \subset X\rangle  
 \quad \mbox{and} \quad \rr_{\cp}(X) := \cR_{\cp}(X)\rtimes X \,.
\end{equation}
Thus $\cR_{\cp}(X)$ is the norm-closed subalgebra of $\cbu(X)$
generated by the $\cp(X/Y)$, where $Y$ runs over the set of {\em all
  linear subspaces} of $X$. Clearly this is a translation invariant
$C^*$-subalgebra of $\cbu(X)$. We shall regard the crossed product
$\rr_{\cp}(X) := \cR_{\cp}(X)\rtimes X$, as a $C^*$-subalgebra of
$\rb(X)$.  Its structure will play a crucial role in what follows. For
instance, for our approach, it will be convenient to assume that
$\cc_0(E)\subset\cp(E)$ and $\cp(0)=\C$. Clearly then $\rr_{\cp}(X)$
contains $C^*(X)$ and $\rk(X)$.

It will be natural to call $\cR_{\cp}(X)$ the \emph{algebra of
  elementary interactions of type $\cp$} and
$\rr_{\cp}(X) := \cR_{\cp}(X) \rtimes X$ the \emph{algebra of $N$-body
  type Hamiltonians with interactions of type $\cp$}. Indeed,
$\rr_{\cp}(X)$ is the $C^*$-algebra of operators on $L^2(X)$ generated
by the resolvents of the self-adjoint operators of the form $h(p)+V$,
with $h:X^*\to\R_+$ continuous and proper, and $V\in\cR_{\cp}(X)$
\cite[Proposition 3.3]{GI3}.  The self-adjoint operators affiliated to
$\rr_{\cp}(X)$ will be called \emph{$N$-body Hamiltonians with
  interactions of type $\cp$}.

We give three examples of possible choices for $\cp$, in increasing 
order of difficulty.

First, the ``standard'' $N$-body situation, as described for example
in \cite[Section\ 4]{DG1} and \cite[Section\ 6.5]{GI3}, corresponds to
the choice $\cp(E)=\maC_0(E)$. The algebra of elementary interactions
$\cR_{\cp}(X) = \cR_{\cc_0}(X)$ in this case has a remarkable feature:
it is graded by the ordered set of all linear subspaces of $X$, more
precisely $\cR_{\cc_0}(X)$ is the norm closure of
$\sum_{Y \subset X}\maC_0(X/Y)$, this sum is direct, and we have
$\maC_0(X/Y) \maC_0(X/Z)\subset \maC_0(X/(Y\cap Z)$. Then the
corresponding algebra $\rr_{\cc_0}(X)$ of $N$-body Hamiltonians with
interactions of type $\maC_0$ inherits a graded $C^*$-algebra
structure \cite{Mageira,Mageira2}. The usual $N$-body Hamiltonians are
self-adjoint operators affiliated to $\rr_{\maC_0}(X)$, and their
analysis is greatly simplified by the existence of the grading.

Let us now discuss the choice of the space of potential functions
$\cp(X/Y)$ that will used in this paper. Namely, for any real finite
dimensional vector space $E$ we consider the spherical
compactification $\overline{E}$ of $E$ and denote $\overline{\maC}(E)
= \maC(\overline{E})$.  Our main goal in this paper is to treat the
larger class of interactions $\cp(E) = \overline{\maC}(E)$ and to
analyze the $N$-body Hamiltonians associated to them. We recall the
notations already used in the introduction:
\begin{align}
 &\ce(X) \, := \, \cR_{\overline{\maC}}(X) \, := \, \langle\,
 \overline{\maC}(X/Y),\, Y \subset X \rangle \, \subset \, \cbu(X)\,,
 \label{eq:mainNalgintro} \\
 &\re(X) \, := \, \rr_{\overline{\maC}}(X) \, :=\, \ce(X)\rtimes X 
\, \subset \, \rb(X) \,.
\label{eq:mainNalgintro-bis}
\end{align}
One of the main difficulties now comes from the absence of a grading
of the algebra $\ce(X)$ of elementary interactions, which requires
more care in understanding its spectrum.  Observe that, besides the
ideal $\maC_0(X)\rtimes X \simeq \rk(X)$ of compact operators,
$\re(X)$ also contains the spherical algebra
$\rs(X) :=\maC(\oX)\rtimes X$ consisting of two-body type operators.

A third natural choice, which gives an even larger class of elementary
interactions and of $N$-body type Hamiltonians, is to take as $\cp(E)$
the algebra of slowly oscillating functions on $E$, a class of
functions whose importance has been pointed out by H.O.\ Cordes (see
Section 6.2 in \cite{GI3} for a discussion of this point and several
references). In this context, we mention M.E.\ Taylor's thesis
\cite{TaylorTAMS1971} where hypoelliptic operators with slowly
oscillating coefficients of two-body type are considered: this is one
of the first papers where Fredholmness criteria are obtained in a
general setting by using the comparison $C^*$-algebras introduced by
Cordes. In fact, his $C^*$-algebra $\mathfrak{A}$ is just the crossed
product of the $C^*$-algebra of slowly oscillating functions by the
action of $X$.

\subsection{The algebra of elementary interactions}\label{ss:elint}  

The algebra $\ce(X)$ will play a leading role in our approach.  From
the definition, it follows that $\ce(X)$ is a translation invariant
subalgebra since the generating subspaces $\maCXY$ are already
translation invariant.  The algebra $\ce(X)$ is not graded, as in the
standard $N$-body framework of the algebra $\cR_{\cc_0}(X)$, but has a
natural filtration that plays an important role in our analysis.

Let us fix a linear subspace $Z\subset X$. Then $X/Z$ is a finite
dimensional real vector space, and hence the $C^*$-algebra
$\ce(X/Z)\subset\cbu(X/Z)$ is well defined and the embedding
$\cbu(X/Z)\subset\cbu(X)$ allows us to identify $\ce(X/Z)$ with a
$C^*$-subalgebra of $\ce(X)$. If $Y \supset Z$ is another linear
subspace then $Y/Z\subset X/Z$ and we may identify $X/Y =
(X/Z)/(Y/Z)$. Therefore we can identify
\begin{equation}\label{eq:maiNsalg}
  \ce(X/Z) \, = \, \text{ $C^*$-subalgebra of } \ce(X) \text{
    generated by }  \ccup_{Z \subset Y}\,  \maCXY \,.
\end{equation}
Thus, the $C^*$-algebra $\ce(X)$ is equipped with a family of
$C^*$-subalgebras $\ce(X/Y)$, where $Y$ runs over the set of linear
subspaces of $X$, such that, for $0\subset Z\subset Y\subset X$, we
have
\begin{equation}\label{eq:cexyz}
  \C \, = \, \ce(0) \, = \, \ce(X/X) \, \subset \, \ce(X/Y) \, \subset
  \, \ce(X/Z) \, \subset \, \ce(X) \, .
\end{equation}
Recall now that $\SS_X$ consists of the half-lines of $X$. We shall
denote by $[\alpha]$ the one dimensional subspace generated by a
half-line $\alpha \in \SS_{X}$. Observe that the algebras
$\ce(X/[\alpha])$ are maximal among the non-trivial subalgebras of
$\ce(X)$ of the form $\ce(X/Y)$.

Translation at infinity along a direction
$\alpha = \RR_{+} a \in \SS_X$ gives us a linear projection
$\tau_\alpha$ of $\ce(X)$ onto the subalgebra $\ce(X/[\alpha])$ as
follows. For $u \in \ce(X)$ we define
\begin{equation}\label{eq:cpa}
   \tau_\alpha (u)(x) \,:= \, \displaystyle{\lim_{r\to+\infty}}
   u(ra+x)\,.
\end{equation}

\begin{lemma} \label{lemma.tau}
Let $Y \subset X$ be a real, linear subspace and $u \in \maCXY$. Then
 \begin{equation*}
  \tau_{\alpha}(u) \,= \ 
 \begin{cases}
  u(\pi_Y(\alpha)) \in \CC & \mbox{ if }\  \alpha \not\subset Y\\
  \ u & \mbox{ if } \ \alpha \subset Y \,.
 \end{cases}
\end{equation*}
\end{lemma}

\begin{proof}
 If $\alpha \not\subset Y$, $\pi_Y(\alpha)$ is a half line in $X/Y$,
 and hence $u(\pi_Y(\alpha))$ is defined. The fact that the limit is
 as stated follows from the definition.
\end{proof}

Note that, in the above lemma, $\tau_\alpha(u)$ is a constant if $\alpha
\not \subset Y$.  The lemma gives right away the following.

\begin{proposition}\label{pr:pe}
  If $\alpha \in\SS_X$ and $u \in \ce(X)$, then the limit in
  \eqref{eq:cpa} exists for all $x\in X$, is independent of the choice
  of $a \in \alpha$, and $\tau_\alpha(u)\in\ce(X)$. The map
  $\tau_\alpha:\ce(X)\to\ce(X)$ is an algebra morphism with range
  $\ce(X/[\alpha])$ and $\tau_\alpha(u)=u$ for all
  $u \in \ce(X/[\alpha])$.
\end{proposition}

\begin{proof} 
  Lemma \ref{lemma.tau} shows that the map $\tau_\alpha$ maps $\maCXY$
  to itself if $\alpha \subset Y$, and maps $\maCXY$ to $\CC$
  otherwise. The subspace of $B \subset \ce(X)$, for which the limit
  $\tau_{\alpha}(u)(x)$ exists for any $x$ is a norm closed,
  conjugation invariant subalgebra of $\ce(X)$. Since $B$ contains the
  generators of $\ce(X)$, we obtain that $B = \ce(X)$.  Consequently,
  the limit $\tau_{\alpha}(u)(x)$ exists for all $u \in \ce(X)$ and
  all $x \in X$. Also, we obtain that $\tau_{\alpha}$ maps the
  generators of $\ce(X)$ to a system of generators of
  $\ce(X/[\alpha]) \subset \ce(X)$, and hence $\tau_{\alpha}$ maps
  $\ce(X)$ onto $\ce(X/[\alpha])$ surjectively.  To complete the
  proof, we notice that
  $\tau_{\alpha} \circ \tau_{\alpha} = \tau_{\alpha}$ on the standard
  system of generators of $\ce(X)$, and hence $\tau_{\alpha} = id$ on
  the range of $\tau_{\alpha}$, that is, on $\ce(X/[\alpha])$.
\end{proof}

\begin{remark}\label{re:pab}{\rm 
The proof of Proposition \ref{pr:pe} gives that,} 
for each $\alpha\in\SS_X$, \emph{the relation \eqref{eq:cpa}
  defines a unital endomorphism $\tau_\alpha$ of $\ce(X)$, which is
  also a linear projection of $\ce(X)\,$ onto the subalgebra
  $\ce(X/[\alpha])$. We note that $\tau_\alpha$ does not commute
with $\tau_\beta$ in general: if a subspace $Z$ does not contain
$\alpha$ and $\beta$ and $u\in \maC(\overline{X/Z})$ then
$\tau_\alpha\tau_\beta(u)= u(\pi_Z(\beta))$ and $\tau_\beta
\tau_\alpha(u)= u(\pi_Z(\alpha))$.}
\end{remark}

\begin{remark}\label{re:function}{\rm 
For the purpose of this paper, the elements of $\ce(X)$ should be
thought as multiplication operators on the space $L^2(X)$. If,
according to the notational conventions from the beginning of Section
\ref{s:locinf}, we denote by $u(q)$ the operator of multiplication by
$u\in\ce(X)$ and, if we set $\tau_\alpha(u(q))=\tau_\alpha(u)(q)$,
then we get an expression similar to \eqref{eq.def.Pa}:
\begin{equation}\label{eq:cpaop}
  \tau_\alpha(u(q)) \, = \, \slim_{r\to+\infty} T_{ra}^*u(q)T_{ra}
  \, = \, \slim_{r\to+\infty} u(ra+q) \,. 
\end{equation} 
We emphasize however that $\slim_{a\to\alpha} T_{a}^*u(q)T_{a}$
\emph{does not exist} for general $u \in \ce(X)$.  }
\end{remark}

The next few results concern the subalgebras $\ce(X/Y)$.

\begin{proposition}\label{pr:cey}
  Let $\bar\alpha=(\alpha_1, \alpha_2, \ldots, \alpha_n)$ be a system
  of half-lines, which generate a subspace $Y$ of $X$. Then
\begin{equation}\label{eq:ygen}
  \ce(X/Y) \, = \, \ce(X/[\alpha_1]) \cap \dots \cap \ce(X/[\alpha_n]) .
\end{equation}
The morphism $\tau_{ \bar\alpha} :=
\tau_{\alpha_1}\tau_{\alpha_2}\ldots\tau_{\alpha_n}$ is a linear
projection of $\ce(X)$ onto $\ce(X/Y)$.
\end{proposition}

\begin{proof}
  If $u \in \maC(\overline{X/Z})$, for some $Z$, then Lemma
  \ref{lemma.tau} gives $\tau_{ \bar\alpha}(u)=u $ if $Y \subset Z$
  and $\tau_{ \bar\alpha}(u) \in \CC\ $ otherwise. In any case,
  $\tau_{ \bar\alpha}(u) \in \ce(X/Y)$. Since $\tau_{ \bar\alpha}$ is
  a morphism, the range of $\tau_{ \bar\alpha}$ is
  included in $\ce(X/Y)$ and $\tau_{ \bar\alpha}(u)=u$ if
  $u\in\ce(X/Y)$. Thus $\tau_{ \bar\alpha}$ is a linear projection of
  $\ce(X)$ onto $\ce(X/Y)$. Let $u\in\ce(X)$. We obtain that $u\in\ce(X/Y)$
  if, and only if, $\tau_{ \bar\alpha}(u)=u$. If $u$ belongs to the
  right hand side of \eqref{eq:ygen}, then $\tau_{ \bar\alpha}(u)=u$,
  so $u\in\ce(X/Y)$.
\end{proof}

Note that a permutation of the $\alpha_1, \alpha_2, \ldots, \alpha_n$
will give a different projection onto $\ce(X/Y)$ (see Remark
\ref{re:pab}). More generally, if $\bar\beta=(\beta_1,\dots,\beta_m)$
is a second system of half-lines that generates $Y$, then
$\tau_{\bar\beta}$ is a projection $\ce(X)\to\ce(X/Y)$ distinct from
$\tau_{\bar\alpha}$ in general.

By using \eqref{eq:cexyz} and \eqref{eq:ygen} we get
\begin{equation}\label{eq:ygen2}
  \ce(X/Y) \, = \, \ccap_{\alpha\subset Y}\ce(X/[\alpha]) \, = \,
  \{u\in\ce(X) , \, \tau_{\alpha}(u) = u \ \forall\,\alpha \subset Y
  \} \,,
\end{equation}
from which we get
\begin{equation}\label{eq:cexy2}
  \ce(X/Y) = \{u\in\ce(X) , \, u(x+y)=u(x) \ \forall\, y\in Y\} =
  \ce(X)\cap\cbu(X/Y).
\end{equation}
Indeed, if $C$ is the middle term in \eqref{eq:cexy2}, then $\ce(X/Y)
\subset C$, by the definition of $\ce(X/Y)$ and the definition of
$\tau_\alpha$ shows that $C$ is included in the right hand side of
\eqref{eq:ygen2}.

\begin{proposition}\label{pr:sum}
If $Y,Z$ are subspaces of $X$ then $\ce(X/(Y+Z))=\ce(X/Y) \cap
\ce(X/Z)$.
\end{proposition}

\begin{proof}
  Let $Y',Z'$ be supplements of $Y\cap Z$ in $Y$ and $Z$
  respectively. Choose a basis $a_1,\dots,a_n$ of $Y+Z$ such that
  $a_1,\dots,a_i$ is a basis of $Y'$, then $a_{i+1},\dots,a_j$ is a
  basis of $Y\cap Z$, and $a_{j+1},\dots,a_n$ is a basis of
  $Z'$. Denote $\alpha_k$ the half-line determined by $a_k$. From
  \eqref{eq:ygen} we get
\[
\ce(X/Y) \, = \, \cap_{k<j}\ce(X/[\alpha_k]) \quad\text{and}\quad
\ce(X/Z)  \, =  \, \cap_{k>i}\ce(X/[\alpha_k]) \,,
\]
and hence $\ce(X/Y)\cap\ce(X/Z)=\cap_{k=1}^n\ce(X/[\alpha_k])$, which is
equal to $\ce(X/(Y+Z))$, by \eqref{eq:ygen}. This completes the proof.
\end{proof}

\subsection{The character space}\label{ss:prelim}

We now turn to the study of the spectrum (or character space) of the
algebra $\ce(X)$ of elementary interactions.  We begin with an
elementary remark.

Let $x \in X$.  Then to $x$ there corresponds the character $\chi_x(u)
= u(x)$ on $\cbu(X)$. The character $\chi_x$ is completely determined
by its restriction to the ideal $\maC_0(X)$ of $\cbu(X)$.  Similarly,
if $\alpha \in \oX$, then $\alpha$ defines a character $\chi_\alpha :
\maCX \to \CC$ by $\chi_\alpha(u) = u(\alpha)$.

The following lemma and its corollary will provide a crucial
ingredient in the proof of Theorem \ref{th:char} identifying the
spectrum of $\ce(X)$, which is one of our main results.

\begin{lemma}\label{lemma.quotient}
  Let $ Y \subset X$ be a subspace, let $B$ be the $C^*$--algebra
  generated by $\maCX$ and $\maCXY$ in $\cbu(X)$, and let $\alpha \in
  \SS_X \smallsetminus \SS_Y$. Then the character $\chi_\alpha$ of
  $\maCX$ extends to a unique character of $B$. This extension is the
  restriction of $\tau_{\alpha}$ to $B$.
\end{lemma}

\begin{proof} 
  Recall that the canonical projection $\pi_{Y} : X \to X/Y$ extends
  to a continuous map $\pi_{Y} : \oX \smallsetminus \SS_Y \to
  \overline{X/Y}$, which sends $\SS_X\smallsetminus\SS_Y$ onto
  $\SS_Y$. Thus $\beta := \pi_Y(\alpha)\in\SS_Y$ and $\chi_\beta$ is a
  character of $\maCXY$. Let $\chi$ be a character of $B$ such that
  $\chi \vert_{\maCX} = \chi_\alpha$. We shall verify now that $\chi
  \vert_{\maCXY} =\chi_\beta$.

  To prove that $\chi \vert_{\maCXY} =\chi_\beta$, it suffices to show
  that the kernel of $\chi_\beta$ is included in that of $\chi$, which
  means that for $u\in\maCXY$ with $u(\beta)=0$, we should have
  $\chi(u)=0$. By a density argument, it suffices to assume that $u=0$
  on a neighborhood $V$ of $\beta$ in $\overline{X/Y}$. It is clear
  that we can find $v\in\maCX$ with $v(\alpha)=1$ with support in the
  $\pi_Y^{-1}(V)$, and hence $uv=0$. Since $u,v\in B$, we have
\begin{equation*}
   0 = \chi(uv) = \chi(u) \chi(v) = \chi(u) \chi_{\alpha}(v) =
   \chi(u) v(\alpha) = \chi(u) \,.
\end{equation*}
This proves that $\chi \vert_{\maCXY} =\chi_\beta$, as claimed.
  
  From the relation $\chi \vert_{\maCXY} =\chi_\beta$ just proved, we
  obtain the uniqueness of $\chi$, since $\maCX$ and $\maCXY$ generate
  $B$.  To complete the proof, let us notice that the restriction of
  $\tau_{\alpha}$ to $\maCX$ is $\chi_\alpha$ and its restriction to
  $\maCXY$ is also character, because $\alpha \not\subset Y$. Thus
  $\tau_\alpha$ is a character on $B$ and we get
  $\chi=\tau_{\alpha}\vert_{B}$ by uniqueness.  This completes the
  proof.
\end{proof}

\begin{corollary}\label{cor.equal}
  Let $\chi_1$ and $\chi_2$ be characters of $\ce(X)$. Let us assume
  that there exists
  $\alpha \in \SS_{X}$ such that $\chi_1(u) = \chi_2(u)= u(\alpha)$
  for all $u \in \maCX$ and that $\chi_1 = \chi_2$ on
  $\ce(X/[\alpha])$.  Then $\chi_1 = \chi_2$. 
\end{corollary}

\begin{proof} It is enough to show that $\chi_1 = \chi_2$ on each of
  the algebras $\maCXY$, since the later generate $\ce(X)$, by
  definition.  Since $\chi_1 = \chi_2= \chi_\alpha$ on $\maCX$, we
  obtain $\chi_1 = \chi_2$ on all $\maCXY$ with $\alpha \not\subset
  Y$, by Lemma \ref{lemma.quotient}.  Since $\maC(X/[\alpha])$
  contains (indeed, it is generated by) all $\maCXY$ with $\alpha
  \subset Y$, the result follows.
\end{proof}

We now proceed to the construction of the characters of $\ce(X)$.  We
begin with a remark concerning the simplest nontrivial case that helps
to understand the general case.

\begin{remark}\label{re:doi} {\rm 
If $\alpha\in\SS_X$ and $\beta\in\SS_{X/[\alpha]}$, then $[\beta]$ is
the one dimensional subspace generated by $\beta$ in $X/[\alpha]$, and
hence $\pi_{[\alpha]}^{-1}([\beta])$ is a two dimensional subspace of
$X$ that we shall denote by $[\alpha,\beta]$. Note that we may and
shall identify $(X/[\alpha])/[\beta]$ with $X/[\alpha,\beta]$. Then
Proposition \ref{pr:pe} gives us two morphisms
$\tau_\alpha:\ce(X)\to\ce(X/[\alpha])$ and
$\tau_\beta:\ce(X/[\alpha])\to\ce(X/[\alpha,\beta])$ that are linear
projections. Thus
$\tau_\beta\tau_\alpha:\ce(X)\to\ce(X/[\alpha,\beta])$ is a morphism
and a projection, and if $a\in X/[\alpha,\beta]$, then
$u\mapsto(\tau_\beta\tau_\alpha u)(a)$ is a character of $\ce(X)$.  }
\end{remark}

We now extend the construction of the above remark to an arbitrary
number of half-lines. However, it will be convenient first to
introduce the following notations.
\begin{notations}{\rm Our construction involves finite sequences
    $\overrightarrow\alpha := (\alpha_1, \alpha_2, \ldots, \alpha_n)$
    with $0\leq n\leq\dim(X)$ and linear subspaces
    $[\overrightarrow\alpha] := [\alpha_1, \alpha_2, \ldots,
    \alpha_n]$
    of $X$ associated to them. If $n=0$, then we define
    $\overrightarrow\alpha$ as the empty set and we associate to it
    the subspace of $X$ reduced to zero: $[\emptyset]=\{0\}$. If $n=1$
    then $\overrightarrow\alpha=(\alpha_1)$ with $\alpha_1\in\SS_X$
    and, as before, $[\alpha_1]$ is the one dimensional subspace of
    $X$ generated by $\alpha_1$.  The case $n=2$ is treated in the
    Remark \ref{re:doi} and we extend the notation to $n\geq3$ by
    induction: $\alpha_n\in \SS_{X/[\alpha_1,\dots,\alpha_{n-1}]}$ and
    $[\alpha_1,\dots,\alpha_{n}]=\pi_Y^{-1}([\alpha_n])$ is an
    $n$-dimensional subspace of $X$ (here
    $Y=[\alpha_1,\dots,\alpha_{n-1}]$).  Note that we may identify
    $X/[\alpha_1,\dots,\alpha_{n}]
    =\big(X/[\alpha_1,\dots,\alpha_{n-1}]\big)/[\alpha_n]$.
    We denote $\tilde\Omega_X\sp{(n)}$ the set of the just defined
    finite sequences $\overrightarrow\alpha$ of length $n$ and
\begin{equation*}
  \Omega_X\sp{(n)} \, := \, \{(a, \overrightarrow\alpha ) ,\,
  \overrightarrow\alpha = (\alpha_1, \alpha_2, \ldots, \alpha_n) \in
  \tilde \Omega_X\sp{(n)} \,,\ a \in X/[\alpha_1, \ldots, \alpha_n] \,
  \} \,.
\end{equation*}
In particular, $\Omega_X\sp{(0)} \equiv X$ and $\Omega_X\sp{(N)}
\equiv \tilde \Omega_X\sp{(N)}$ if $N = \dim(X)$, since $[\alpha_1,
  \ldots, \alpha_N] = X$.  Let
\begin{equation}\label{eq:chr}
\Omega_X=\ccup_{n=0}^{\dim(X)} \, \Omega_X\sp{(n)} \,. 
\end{equation}
}
\end{notations}

\begin{definition}\label{de:tmorph}
If $(a, \overrightarrow\alpha) \in \Omega_X\sp{(n)}$, then we define
\begin{equation}\label{eq:tauchi}
  \tau_{\overrightarrow\alpha} \, = \, \tau_{\alpha_n} \tau_{\alpha_{n-1}}
\dots \tau_{\alpha_1} \quad \mbox{ and } \quad
  \tau_{a, \overrightarrow\alpha} \, = \, \tau_a\tau_{\overrightarrow\alpha} \,,
\end{equation}
which are endomorphisms of $\ce(X)$. We agree that $\tau_\emptyset$ is
the identity of $\ce(X)$.
\end{definition}

In particular, the range of $\tau_{\overrightarrow\alpha}$ is
$\ce(X/[\overrightarrow\alpha])$ and $\tau_{\overrightarrow\alpha}$ is
an endomorphism of $\ce(X)$ and a linear projection of $\ce(X)$ onto
the subalgebra $\ce(X/[\overrightarrow\alpha])$. The morphisms of the
form $\tau_{\bar\alpha}$ considered in Proposition \ref{pr:cey} also
have these properties, but they may be distinct from the
$\tau_{\overrightarrow\alpha}$, the objects $\bar\alpha$ and
$\overrightarrow\alpha$ being different in nature. Note also that,
since $a\in X/[\overrightarrow\alpha]$, translation by $a$ is a
morphism $\tau_a$ of $\ce(X/[\overrightarrow\alpha])$, and hence
$\tau_{a,\overrightarrow\alpha }$ is well defined.

We now introduce what will turn out to be a parametrization of the
characters of $\ce(X)$.

\begin{definition}\label{def.car}
If $(a, \overrightarrow\alpha) \in \Omega_X$, we define the character
$\chi_{a, \overrightarrow\alpha}$ of $\ce(X)$ by the formula
\begin{equation}\label{eq:tauchar} 
  \chi_{a, \overrightarrow\alpha} (u) \, := \,
  \chi_a(\tau_{\overrightarrow\alpha}(u)) \, = \,
  \tau_{\overrightarrow\alpha}(u)(a) \, .
\end{equation}
\end{definition}

We need to explain what happens in the limit case $n=\dim(X)$.

\begin{remark}{\rm
Let $n=\dim(X)$ and $(a, \overrightarrow\alpha) \in \Omega_X\sp{(n)}$.
Then $[\overrightarrow\alpha]=X$, and hence
$X/[\overrightarrow\alpha]=0$, so the only possible choice for $a$ is
$a=0$. Moreover, $\tau_{\overrightarrow\alpha}:\ce(X)\to\C$ is already
a character. Since $\tau_0=\mathrm{id}$, we get
$\chi_{0,\overrightarrow\alpha}=\tau_{\overrightarrow\alpha}$.
}\end{remark}

We are ready now to prove one of our main results, which is a
description of all the characters of the algebra $\ce(X)$.  Recall
that we denote by $\hat{\ce(X)}$ the character space of $\ce(X)$.

\begin{theorem}\label{th:char}
The map $\Omega_X\to \hat{\ce(X)}$ defined by
$(a,\overrightarrow\alpha)\mapsto \chi_{a, \overrightarrow\alpha}$ is
bijective.
\end{theorem}

\begin{proof} 
  The preceding construction shows that $\chi_{a,
    \overrightarrow\alpha}$ is a character, therefore we only need to
  show that every character $\chi$ of $\ce(X)$ is of this form and
  that the pair $(a, \overrightarrow\alpha)$ is uniquely
  determined. To this end, we look at the restriction of $\chi$ to the
  subalgebra $\maCX$ and proceed by induction on the dimension of $X$.

  Every character of $\maCX$ is of the form $u\mapsto u(x) = \chi_{x}$
  for some $x \in \oX$. Hence there is a unique $x \in \oX$ such that
  $\chi\vert_{\maCX} = \chi_x$.  We distinguish two cases: $x \in X$
  and $x \in \oX \smallsetminus X$. In the first case, we have
  $x=a \in X$; that is, $\chi(u) = u(a)$ for all $u \in \ce(X)$. In
  our terminology, this means $\chi=\chi_{a,\emptyset}$. The
  characters $\chi$ of this form are characterized by the fact that
  the restriction of $\chi$ to $\maC_0(X)$ is non-zero. The value of
  $a$ is then determined by restriction to $\maC_0(X)$, since there is
  a one-to-one correspondence between the characters of $\maC_0(X)$
  and the points of $X$.  Thus all the characters
  $\chi_{a, \emptyset}$, $a \in X$, are distinct.

  Now let us assume that $x \notin X$,
  that is, $x = \alpha \in \SS_X := \oX \smallsetminus X$, and that
  the assertion of the theorem is true for all vector spaces of
  dimension strictly less than that of $X$ (induction
  hypothesis). Then the theorem holds for the space $X/[\alpha]$, so
  there is $\overrightarrow\beta=(\beta_1,\dots, \beta_k)$ with
  \begin{equation*}
    \beta_1\in X/[\alpha], \ \beta_2\in X/[\alpha, \beta_1], \ \ldots
    , \ \beta_k \in X/[\alpha, \beta_1, \ldots, \beta_{k-1}] \,,
  \end{equation*}
  such that the restriction of $\chi$ to $\ce(X/[\alpha])$ is given by
  $\chi(u)=(\tau_{\overrightarrow\beta} u)(b)$ for some $b\in
  (X/[\alpha])/\overrightarrow\beta$.  That is, $\chi = \chi_{b,
    \overrightarrow\beta}$ on $\ce(X/[\alpha])$.  Let $a = b$ and let
  $\overrightarrow\alpha$ be obtained by including $\alpha$ in front
  of the sequence $\overrightarrow\beta$, more precisely
  $\overrightarrow\alpha=(\alpha,\beta_1,\dots, \beta_k)$.  Then
  $\chi_{a,\overrightarrow\alpha}(u)
  =(\tau_{\overrightarrow\beta}\circ\tau_{\alpha}u)(b)$ and the
  characters $\chi$ and $\chi_{a, \overrightarrow\alpha}$ coincide on
  $\ce(X/[\alpha])$. On the other hand, on $\maCX$, the characters
  $\chi$ and $\chi_{a, \overrightarrow\alpha}$ coincide with the
  character $\chi_\alpha : \ce(X/[\alpha]) \to \CC$.  Therefore $\chi
  = \chi_{a, \overrightarrow\alpha}$ by Corollary \ref{cor.equal}.
  
  The same argument can be used to show that we obtain a one-to-one
  parametrization of all these characters.  We shall proceed once more
  by induction on the length of $\overrightarrow \alpha$.  If
  $\chi_{a, \overrightarrow \alpha} = \chi_{b, \overrightarrow
    \beta}$, we have two possibilities: first that their restrictions
  to $\maC_0(X)$ is non-zero and, second, that their restrictions to
  $\maC_0(X)$ is zero. In the first case, we must have
  $\overrightarrow \alpha = \emptyset$ and $\overrightarrow \beta =
  \emptyset$, by the discussion earlier in the proof.  By restricting
  to $\maC_0(X)$, we also obtain $a = b \in X$.  Let us assume that
  $\overrightarrow \alpha \neq \emptyset$, then $\chi_{a,
    \overrightarrow \alpha}$ restricts to zero on $\maC_0(X)$, and
  hence $\overrightarrow \beta \neq \emptyset$ as well. Since the
  restrictions of $\chi_{a, \overrightarrow \alpha}$ and $\chi_{b,
    \overrightarrow \beta}$ to $\maCXY$ are $\chi_{\alpha_1}$ and
  $\chi_{\beta_1}$ respectively, we obtain $\alpha_1 = \beta_1$. The
  proof is completed by induction using the restrictions of these
  characters to $\ce(X/[\alpha_1])$, as in the first part of the
  proof.
\end{proof}

We shall describe now the morphism $\tau_\chi$ on $\ce(X)$ defined as
the translation by a character $\chi =
\chi_{a,\overrightarrow\alpha}\in\hat{\ce(X)}$, see Section
\ref{s:locinf}, Definition \ref{def:trinf}.
  
\begin{theorem}\label{thm:morphism}
  The translation morphism associated to the character
  $\chi_{a,\overrightarrow\alpha}$ by Definition \ref{def:trinf} is
  the unital endomorphism $\tau_{a,\overrightarrow\alpha}$ of $\ce(X)$
  introduced in Definition \ref{de:tmorph}.
\end{theorem}

\begin{proof}
If $\chi=\chi_a\equiv\tau_{a,\emptyset}$ for some $a\in X$, then this
is just the usual translation by $a$,
i.e. $\tau_{\chi_a}(u)=\tau_a(u)=a.u$ is the function $x\mapsto
u(a+x)$. In general, we have to use the definition in Definition
\ref{def:trinf}, that is, $(\tau_\chi(u))(y)=\chi(y.u)$ for all $y\in
X$. Thus, if $\chi=\chi_{a,\overrightarrow\alpha}$ as above, then from
Definition \ref{def.car} we get
\[
  (\tau_\chi(u))(x)=\chi_{a,\overrightarrow\alpha}(x.u) \, = \,
  \chi_a(\tau_{\overrightarrow\alpha}(x.u)) \,.
\]
It is clear that $X$ acts by translation on each of the algebras
$\ce(X/Y)$ and that the morphism $\tau_{\overrightarrow\alpha} :
\ce(X)\to\ce(X/\overrightarrow\alpha)$ is covariant for this action,
that is, $\tau_{\overrightarrow\alpha}(x.u) =
x.(\tau_{\overrightarrow\alpha}(u))$.  Thus
\[
(\tau_\chi(u))(x) \,= \, \chi_a(x.(\tau_{\overrightarrow\alpha}(u)))
\, = \, (x.(\tau_{\overrightarrow\alpha}(u)))(a)
\, = \, (\tau_{\overrightarrow\alpha}(u))(x+a)\,,
\]
and hence we get $\tau_\chi(u)=\tau_a\tau_{\overrightarrow\alpha}(u)$,
which is \eqref{eq:tauchi}.
\end{proof}

\begin{remark} {\rm Although we shall not use this here, let us
    mention that in view of Remark \ref{re:cross} and of Theorem
    \ref{thm:morphism}, it is interesting to notice that the action of
    $X$ on the space of characters of $\ce(X)$ is given by
    $\tau_x (\chi_{a, \overrightarrow{\alpha}}) =
    \chi_{a-\pi_{\overrightarrow\alpha}(x), \overrightarrow{\alpha}}$,
    where $\pi_{\overrightarrow\alpha}$ is the canonical map
    $X\to X/[\overrightarrow\alpha]$. Hence, for the determination of
    the essential spectrum, it is enough to consider the characters
    $\chi_{0, \overrightarrow{\alpha}}$ and their associated
    translations
    $\tau_{\chi_{0, \overrightarrow{\alpha}}} = \tau_{0,
      \overrightarrow{\alpha}} = \tau_{\overrightarrow{\alpha}} $.
}\end{remark}

\subsection{The Hamiltonian algebra}\label{ss:ham}
We now apply the results we have proved to the study of essential
spectra.  Since $\ce(X)$ is a translation invariant $C^*$-subalgebra
of $\cbu(X)$ such that $\maC_0(X)+\C\subset\ce(X)$, we may take
$\ca=\ce(X)$ in Section \ref{s:locinf}.  The algebra generated by the
Hamiltonians that are of interest for us is the crossed product
\begin{equation}\label{eq:mainham}
\re(X):=\ce(X)\rtimes X.
\end{equation}
As explained in Section \ref{s:locinf}, $\re(X)$ can be thought as the
closed linear subspace of $\rb(X)$ generated by the operators of the
form $u(q)v(p)$ with $u\in\ce(X)$ and $v\in\maC_0(X^*)$.  On the other
hand, since $\maCXY$ is a translation invariant $C^*$-subalgebra of
$\cbu(X)$, we may also consider the crossed product $\maCXY\rtimes X$
and we clearly have
\begin{equation}\label{eq:mainham1}
  \re(X) = \text{$C^*$-subalgebra of $\rb(X)$ generated by } 
\ccup_{Y \subset X} \maCXY\rtimes X .
\end{equation}
Similarly, for any subspace $Y\subset X$, we may consider the crossed
product $\re(X/Y)=\ce(X/Y)\rtimes X$. We thus obtain a family of
$C^*$-subalgebras of $\re(X)$ that, as a consequence of
\eqref{eq:cexyz}, has the following property: if $Z\subset Y$ then
\begin{equation}\label{eq:rexyz}
   C^*(X) = \re(0) = \re(X/X) \subset\re(X/Y)
   \subset\re(X/Z)\subset\re(X).
\end{equation}
From the general facts described in Section \ref{s:locinf}, and by
taking into account the properties of $\ce(X)$ established in the
preceding subsection, we see that for any $A\in\re(X)$ the map
$x\mapsto \tau_x(A)=T_x^*AT_x$ extends to a strongly continuous map
$\chi\mapsto \tau_\chi(A)\in\re(X)$ on the spectrum of $\ce(X)$ such
that
\[
   \tau_\chi\big(u(q)v(p)\big)= \tau_\chi\big(u(q)\big) v(p) \quad
   \text{for all } u\in\ce(X) \text{ and } v\in\maC_0(X^*).
\]
Here $\chi\in\hat{\ce(X)}$, and hence it is of the form described in
Theorem \ref{th:char} and the associated endomorphism $\tau_\chi$ of
$\ce(X)$ is described in \eqref{eq:tauchi}.  Note that, in virtue of
Theorem \ref{th:infty}, we are only interested in the characters that
belong to the boundary $\delta(\ce(X))$ of $X$ in $\hat{\ce(X)}$,
which are those with $\overrightarrow{\alpha}\neq\emptyset$.  Then
Proposition \ref{pr:locinf} and Theorem \ref{th:char} imply:

\begin{proposition}\label{pr:locinfe}
  Let $\chi=\chi_{a,\overrightarrow\alpha}\in\delta(\ce(X))$. Then
  there is a unique continuous linear map
  $\tau_{a,\overrightarrow\alpha}:\re(X)\to\re(X)$ such that
  $\tau_{a,\overrightarrow\alpha}(u(q)v(p))=
  (\tau_{a,\overrightarrow\alpha}u)(q)v(p)$ for all $u\in\ce(X)$ and
  $v\in\maC_0(X^*)$. This map is a morphism and a linear projection of
  $\re(X)$ onto its subalgebra $\ce(X/[\overrightarrow\alpha]) \rtimes
  X$.
\end{proposition}

Now we shall use the special form of the morphisms
$\tau_{a,\overrightarrow\alpha}$ in order to improve the compactness
criterion of Theorem \ref{th:infty}.

\begin{theorem}\label{th:recomp}
  Let $A\in\re(X)$. Then for each $\alpha\in\SS_X$ and $a\in\alpha$
  the limit $\tau_\alpha(A)\equiv\alpha.A := \slim_{r\to+\infty}
  T_{ra}^*AT_{ra}$ exists and is independent of the choice of $a$. The
  map $\tau_\alpha$ is a morphism and a linear projection of $\re(X)$
  onto its subalgebra $\ce(X/[\alpha]) \rtimes X$.  The operator $A$
  is compact if, and only if, $\tau_\alpha(A)=0$ for all
  $\alpha\in\SS_X$.
\end{theorem}

\begin{proof}
  The first assertion follows from the preceding results, but it is
  easier to prove it directly. Indeed, it suffices to consider $A$ of
  the form $A=u(q)v(p)$ with $u\in\ce(X)$ and $v\in\maC_0(X^*)$. Then
  $T_{ra}^*AT_{ra}=\tau_{ra}(A)=\tau_{ra}(u(q))v(p)$, which converges
  to $(\alpha.u)(q)v(p)$ by Proposition \ref{pr:pe} (or see Remark
  \ref{re:function}). The properties of the endomorphism $\tau_\alpha$
  are consequences of the same proposition. Everything follows also by
  using general properties of crossed products and the fact that at
  the abelian level $\tau_\alpha:\ce(X)\to\ce(X/[\alpha])$ is a
  covariant morphism. To prove the compactness assertion, note first
  that $\tau_\alpha(A)=0$ if $A$ is compact because $T_{ra}\to0$
  weakly as $r\to\infty$. Then if $A\in\re(X)$ and $\tau_\alpha(A)=0$
  for all $\alpha\in\SS_X$ then it is clear by \eqref{eq:tauchi} that
  $\tau_{a,\overrightarrow\alpha}(A)=0$ if
  $\overrightarrow\alpha\neq\emptyset$, and hence $\tau_\chi(A)=0$ for all
  $\chi\in\delta(\ce(X))$, and so $A$ is compact by Theorem
  \ref{th:infty}.
\end{proof}

\begin{remark}\label{re:tcare}{\rm 
If $Y$ is a linear subspace of $X$, then the algebras $\ce(X/Y)$ and
$\re(X/Y)$ are a priori defined by our formalism as algebras of
operators on $L^2(X/Y)$. In Section \ref{ss:elint}, we have defined
$\ce(X/Y)$ as a subalgebra of $\cbu(X)$ satisfying the relation
\eqref{eq:cexy2}; this definition is natural because of our general convention to
identify subalgebras of $\cbu(X/Y)$ with subalgebras of $\cbu(X)$. On
the other hand, we note that the algebras $\re(X/Y)=\ce(X/Y)\rtimes
(X/Y)$ and $\ce(X/Y) \rtimes X$ are quite different objects: indeed
\begin{equation}\label{eq:tensoro}
      \ce(X/Y) \rtimes X\simeq \re(X/Y) \otimes C^*(Y)
\end{equation}
      by a general fact from the theory of crossed products, namely 
\begin{equation}\label{eq:tensor}
  (\ca\otimes\cb) \rtimes (G\times H)\simeq (\ca \rtimes G) \otimes
  (\cb \rtimes H)
\end{equation}
if $(\ca,G)$ and $(\cb,H)$ are amenable $C^*$-dynamical systems. In
particular:
\begin{equation}\label{eq:x/Y}
  \ce(X/[\alpha]) \rtimes X \simeq \re(X/[\alpha]) \otimes
  C^*([\alpha]).
\end{equation} 
}\end{remark}

\begin{corollary}\label{co:quotient}
The map $\tau(A)=\big(\tau_\alpha(A) \big)_{\alpha\in\SS_X}$ induces
an injective morphism
\begin{equation}\label{eq:quotient}
  \re(X)/\rk(X) \, \hookrightarrow \, 
  \pprod_{\alpha\in\SS_X} \, \ce(X/[\alpha]) \rtimes X  \,.
\end{equation} 
\end{corollary}

The following theorem is an immediate consequence of the preceding
corollary.

\begin{theorem}\label{th:ess-last} 
  Let $H$ be a self-adjoint operator on $L^2(X)$ affiliated to
  $\re(X)$.  Then for each $\alpha\in\SS_X$ and $a\in\alpha$ the limit
  $\tau_\alpha(H) \equiv \alpha.H = \slim_{r\to+\infty} T_{ra}^* H
  T_{ra}$ exists and is independent of the choice of $a$. We have
  $\sigma_\ess(H) \, = \, \overline{\cup}_{\alpha\in\SS_X}
  \sigma(\alpha.H)$.
\end{theorem}

The question whether the union $\cup_{\alpha\in\SS_X}
\sigma(\alpha.H)$ is closed or not will not be treated in this paper
(see \cite {NistorPrudhon} for related results).  That the union is
closed if $\re(X)$ is replaced by the standard $N$-body algebra
$\re_{\maC_0}(X)$ is shown in \cite[Theorem 6.27]{GI3} and is a
consequence of the fact that $\{\tau_\alpha(A) , \, \alpha\in\SS_X\}$
is a compact subset of $\re_{\maC_0}(X)$ for each
$A\in\re_{\maC_0}(X)$. Unfortunately this is not true in the present
case.

\begin{lemma}\label{lm:unclosed}
  If $A\in\re(X)$, then $\{\tau_\alpha(A) , \, \alpha\in\SS_X\}$ is a
  relatively compact subset of $\re(X)$, but is not compact in
  general.
\end{lemma}

\begin{proof}
  We first show that $\{\tau_\alpha(A) , \, \alpha\in\SS_X\}$ is a
  relatively compact set in $\re(X)$. Since the product and the sum of
  two relatively compact subsets is relatively compact, it suffices to
  prove this for $A$ in a generating subset of the algebra $\re(X)$,
  so we may assume that $A=u(q)v(p)$ with $u\in\maCXY$ and
  $v\in\maC_0(X^*)$ for some subspace $Y$. Then $\tau_\alpha(A)=A$ if
  $\alpha\subset Y$ and $\tau_\alpha(A)=\tau_\alpha(u)v(p)$ if
  $\alpha\not\subset Y$. In the second case we have
  $\tau_\alpha(u)\in\C$ and $|\tau_\alpha(u)|\leq\|u\|$, so it is
  clear that the set of the $\tau_\alpha(A)$ is relatively compact.

  We shall give now an example when this set is not closed. Let
  $X=\R^2$, $Y=\{0\}\times\R$, and let us identify
  $X/Y=\R\times\{0\}$. The operator $A$ will be of the form
  $A=u(q)v(p)$ so that $\tau_\alpha(A)=\tau_\alpha(u)(q)v(p)$ with
  $u=u_0+u_Y$ for some $u_0\in\maC(\oX)$ and $u_Y\in\maC(\oXY)$.  We
  have $\oXY=\overline{\R\times\{0\}}\equiv[-\infty,+\infty]$, and hence
  $\SS_{\oXY}$ consists of two points $\pm\infty$.  If
  $\alpha\in\SS_X$ then $\tau_\alpha(u)=u_0(\alpha)+\tau_\alpha(u_Y)$
  where $\tau_\alpha(u_Y)=u_Y$ if $\alpha\subset Y$ and
  $\tau_\alpha(u_Y)=u_Y(\pi_Y(\alpha))$ if $\alpha\not\subset Y$.  In
  the last case we have only two possibilities:
  $\tau_\alpha(u_Y)=u_Y(+\infty)$ if $\alpha$ is in the open right
  half-plane and $\tau_\alpha(u_Y)=u_Y(-\infty)$ if $\alpha$ is in the
  open left half-plane.

  Let $\beta$ be the upper half-axis, i.e.\ $\beta=\{(0,y), \, y>0\}$,
  and let us choose $u_0$ such that $u_0(\gamma)\neq u_0(\beta)$ for
  all $\gamma\in\SS_X,\gamma\neq\beta$. Then choose $u_Y$ such that
  $u_Y(+\infty)-u_Y(-\infty)$ be strictly larger than
  $u_0(\gamma)-u_0(\beta)$ for all $\gamma\in\SS_X$.  Then
  $\{\tau_\alpha(u), \, \alpha\in\SS_X\}$ consists of the following
  elements: $u_0(\beta)+u_Y$, $u_0(-\beta)+u_Y$,
  $u_0(\alpha)+u_Y(+\infty)$ if $\alpha$ is in the open right
  half-plane, and $u_0(\alpha)+u_Y(-\infty)$ if $\alpha$ is in the
  open left half-plane. We shall prove that this set is not
  closed. Indeed, let $\{\alpha_n\}$ be a sequence of rays in the open
  right half-plane that converges to $\beta$. Then
  $\tau_{\alpha_n}(u)=u_0(\alpha_n)+u_Y(+\infty)$ is a sequence of
  complex numbers that converges to $u_0(\beta)+u_Y(+\infty)$. This
  number cannot be of the form $\tau_\gamma(u)$ for some
  $\gamma\in\SS_X$ because, if $\gamma\subset Y$, then
  $\tau_\gamma(u)=u_0(\gamma)+u_Y$ is not a number. If $\gamma$ is in
  the open right half-plane, then
  $\tau_\gamma(u)=u_0(\gamma)+u_Y(+\infty)$, which cannot be equal to
  $u_0(\beta)+u_Y(+\infty)$, because $u_0(\gamma)\neq u_0(\beta)$.  On
  the other hand, if $\gamma$ is in the open left half-plane, then
  $\tau_\gamma(u)=u_0(\gamma)+u_Y(-\infty)$, which cannot be equal to
  $u_0(\beta)+u_Y(+\infty)$ because $u_0(\gamma)-u_0(\beta)<
  u_Y(+\infty)-u_Y(-\infty)$.
\end{proof}

\begin{remark}{\rm It is important to notice that finding good
    compactifications of $X$ related to the $N$-body problem is useful
    for the problem of approximating numerically the eigenvalues and
    eigenfunctions of $N$-body Hamiltonians \cite{ACN, Faupin, Flad3,
      Flad1, Fournais1, VasyReg}. In particular, this gives a further
    justification for trying to find the structure of the character
    space of $\ce(X)$.  }
\end{remark}

\subsection{Self-adjoint operators affiliated to 
\texorpdfstring{$\re(X)$}{} } \label{ss:affnb}

Our purpose here is to show that the class of self-adjoint operators
affiliated to $\re(X)$ is quite large. As mentioned at the beginning
of Section \ref{s:nbd}, we may and shall assume  $\dim X\geq2$. 

We first prove Theorem \ref{th:intro1}. We recall the definition of
$\ce^\sharp(X)$ in terms of the algebras $\cb(\oXY)$ defined as in
\eqref{eq:bg} and Lemma \ref{lm:bg}.  Note that, according to our
notational conventions, $\cb(\oXY)$ is identified with a $C^*$-algebra
of functions on $X$.

\begin{definition}\label{df:bex}
  $\ce^\sharp(X)$ is the $C^*$-subalgebra of $L^\infty(X)$ generated
  by the functions of the form $v\circ\pi_Y$, where $Y$ runs over the
  set of linear subspaces of $X$ and $v\in\cb(\oXY)$. 
\end{definition}

\begin{proposition}\label{pr:bgx}
  $u\in\ce^\sharp(X)$ if, and only if, there is a sequence of functions
  $u_n\in\ce(X)$ such that $\sup_n\|u_n\|_{L^\infty(X)}<\infty$ and
  $\lim_n\|u_n(q)-u(q)\|_{\ch^s(X)\to\ch(X)}=0$ for some $s>0$.
\end{proposition}

\begin{proof}
  We will need the following consequence of Proposition \ref{pr:bg}:
  $u\in\cb(\oXY)$ if, and only if, there is a sequence of functions
  $u_n\in\maCXY$ such that $\|u_n\|_{L^\infty}\leq C$ with $C$
  independent of $n$ and $\lim_n\|u_n(q)-u(q)\|_{\ch^s(X)\to\ch(X)}=0$
  for some $s>0$. For the proof, it is useful to distinguish between
  the function $u$ on $X$ and the function $u'$ on $X/Y$ related to it
  by $u=u'\circ\pi_Y$. Then Proposition \ref{pr:bg} gives us functions
  $u'_n:X/Y\to\C$ of class $\maCXY$ such that
  $\|u'_n\|_{L^\infty(X/Y)}\leq C$ and $u'_n(q)\to u'(q)$ in norm
  in the space of bounded operators $\ch^s(X/Y)\to\ch(X/Y)$. Thus, if
  we set $u_n=u'_n\circ\pi_Y$, it suffices to show that
  $\lim_n\|u_n(q)-u(q)\|_{\ch^s(X)\to\ch(X)}=0$.  But this is clear
  because, if $Z$ is a subspace supplementary to $Y$ in $X$, then we
  have $X/Y\simeq Z$, $\ch(X)\simeq\ch(Y)\otimes\ch(Z)$ and
  $\ch^s(X)\simeq\big(\ch^s(Y)\otimes\ch(Z)\big) \cap
  \big(\ch(Y)\otimes\ch^s(Z)\big)$.

  Since $\ce^\sharp$ is the norm closure in $L^\infty$ of the space of
  linear combinations of products of functions in $\cb(\oXY)$ with $Y$
  running over all subspaces of $X$, it remains to prove that if $u$
  is a finite product $u=u^1\dots u^k$ of functions
  $u^i\in\cb(\overline{X/Y_i})$, then one may construct a sequence
  $\{u_n\}$ as in the statement of the proposition. By what we have
  proved, such a sequence $\{u^i_n\}$ exists for each $i$ and clearly
  it suffices to take $u_n=u^1_n\dots u^k_n$.
\end{proof}

\textbf{Proof of Theorem \ref{th:intro1}. } We consider first the
operator $H$ defined in \eqref{eq:Hh}. Since $\dim X\geq2$, the
function $h$ is either lower or upper semi-bounded, and hence we may
assume $h\geq0$.  In Theorem \ref{th:recall} we take $H_0=h(p)$, so
$H_0$ is a positive operator strictly affiliated to $\re(X)$. Then,
according to Theorem \ref{th:recall}, if $V$ is a bounded
self-adjoint operator on $\ch$ such that
\begin{equation}\label{eq:seccond}
\varphi(H_0)V(H_0+1)^{-1/2}\in\re(X) \quad\text{for all }
\varphi\in\cc_\rmc(\R)
\end{equation}
then $H=H_0+V$ is a self-adjoint operator strictly affiliated to
$\re(X)$. Since $h$ is a proper, continuous function we have
$\varphi\circ h\in\cc_\rmc(\R)$, and hence for any $s>0$ the function
$\psi(k)=\varphi(h(k))\jap{k}^s$ also belongs to $\cc_\rmc(\R)$,
so $\psi(p)\in\re(X)$. Then $\varphi(H_0)=\psi(p)\jap{p}^{-s}$ hence
\eqref{eq:seccond} is satisfied if $\jap{p}^{-s}V\in\re(X)$. This last
fact is clearly true if $V=u(q)$ with $u\in\ce(X)$ and remains true if
$u\in\ce^\sharp(X)$ by Proposition \ref{pr:bgx}.

Now let $H$ be the self-adjoint operator associated to the operator
$L:\ch^m\to\ch^{-m}$ defined by \eqref{eq:L}. Then we take
$H_0=1+\sum_{|\mu|=m}p^{2\mu}$ and equip $\ch^m$ with the scalar
product $\braket{u}{H_0u}$. If we set $V=L-H_0$, then \eqref{eq:L1}
implies $V\geq-(1-\delta)H_0-\gamma$. Since $\delta>0$ we see that the
conditions on $V$ in Theorem \ref{th:recall} are satisfied (with a
change of notation). The second condition on $V$ is clearly
satisfied if $H_0^{-1}VH_0^{-1/2}\in\re(X)$. Observe that the operator
$V$ has the same form as $L$, only the coefficients $g_{\mu\nu}$ in
the principal part being changed in an irrelevant manner (replaced by
$g_{\mu\nu}-\delta_{\mu\nu}$ and $g_{00}-1$ respectively). Thus it
remains only to check that $p^\mu H_0^{-1}g_{\mu\nu} p^\nu H_0^{-1/2}$
belongs to $\re(X)$ if $|\mu|,|\nu|\leq m$. Since $p^\nu H_0^{-1/2}$
belongs to the multiplier algebra of $\re(X)$, it suffices to have
$p^\mu H_0^{-1}g_{\mu\nu}\in\re(X)$. Since $H_0$ is of order
$\jap{p}^{2m}$ and $p^\mu$ is of order at most $m$, this follows from
what we proved before in the case $H=h(p)+V$.
\qed

\begin{remark}\label{re:unbcoeff}{\rm One may treat, by the technique
    of the preceding proof, operators $L$ with unbounded coefficients
    in the terms of lower order. Assume that for each $\mu,\nu$ the
    operator of multiplication by $g_{\mu\nu}$ maps $\ch^{m-|\nu|}$
    into $\ch^{|\mu|-m}$. Then $L:\ch^m\to\ch^{-m}$ is well defined
    and the condition \eqref{eq:L1} ensures the existence of the
    self-adjoint operator $H$ associated to it.  It has been shown in
    \cite[Example 4.13]{GI3} that this operator is affiliated to the
    crossed product $\cbu(X)\rtimes X$, and hence its essential spectrum
    can be described in terms of localizations at infinity of
    $H$. However, its affiliation to the smaller algebra $\re(X)$
    would give a much more precise characterization of the essential
    spectrum.  For this, by the argument of the preceding proof, it
    suffices to have
    $p^\mu H_0^{-1}g_{\mu\nu} p^\nu H_0^{-1/2}\in\re(X)$ for all
    $\mu,\nu$. And this is satisfied if the operator $g_{\mu\nu}(q)$
    is the norm limit in $\cb(\ch^{m-|\nu|},\ch^{|\mu|-m-1})$ of a
    sequence of operators $g^k_{\mu\nu}(q)$ with
    $g^k_{\mu\nu}\in\ce(X)$.  }\end{remark}

In the rest of this section we consider only potentials that have a
simpler $N$-body type structure, as explained in Subsection
\ref{page:nbody} (page \pageref{page:nbody}), and we shall prove 
Theorems \ref{th:ess3-intro} and
\ref{th:ess4-intro}.  We will be able to cover a large class of such
interactions by using a more explicit description of the algebras
$\maCXY\rtimes X$ that we describe now.

Observe first that if $Z$ is a supplement of $Y$ in $X$, so $Z$ is a
linear subspace of $X$ such that $Y\cap Z=\{0\}$ and $Y+Z=X$, then:
\begin{equation}\label{eq:yz}
  \maCXY\rtimes X = C^*(Y) \otimes \rs(Z) \quad\text{relatively to }
  L^2(X)=L^2(Y)\otimes L^2(Z) .
\end{equation}
Indeed, $\maCXY\rtimes X$ is the norm closed subspace generated by the
operators of the form $u(q)v(p)$ with $u\in \maCXY$ and $v(p)\in
C^*(X)$. But once $Z$ is chosen, we may identify $\maCXY=1\otimes
\maC(\overline{Z})$ and $C^*(X)=C^*(Y)\otimes C^*(Z)$, and hence
\eqref{eq:yz}. Of course, this is a particular case of the relation
\eqref{eq:tensor} from Remark \ref{re:tcare}.

It is useful to express \eqref{eq:yz} in an intrinsic way, independent
of the choice of $Z$. This is in fact an extension of Theorem
\ref{th:sphalg} to the present setting.

Observe first that if $A$ is a bounded operator on $L^2(X)$ and
$[A,T_y]=0$ for all $y\in Y$, then $T^*_xAT_x$ depends only on the
class $z=\pi_Y(x)$ of $x$ in $X/Y$. Thus we have an action $\tau$ of
$X/Y$ on the set of operators $A$ in the commutant of $\{T_y\}_{y\in
  Y}$ such that $\tau_z(A)=T^*_xAT_x$ if $\pi_Y(x)=z$.  Later on we
shall keep the notation $\tau_a(A)=T^*_aAT_a$ for $a\in X/Y$ since the
correct interpretation should be clear from the context.

\begin{theorem}\label{th:sphalgY} The set
  $\maCXY\rtimes X$ consists of the operators $A \in \rb(X)$ that have
  the position-momentum limit property and are such that
 \begin{enumerate}[(i)]
\item $[A,T_y]=0$ for all $y\in Y$,
\item for each $\alpha\in\SS_{X/Y}$ the limit $\slim \tau_a(A)^{(*)}$
  with $a\to\alpha$ in $X/Y$ exists.
\end{enumerate}
\end{theorem}

\begin{proof}
  Let $\check{\alpha}=\pi_Y^{-1}(\tilde\alpha)$ be the inverse image
  of the filter $\tilde\alpha$ through the map $\pi_Y$, i.e.\ the set
  of subsets of $X$ of the form $\pi_Y^{-1}(F)$ with
  $F\in\tilde\alpha$. This is a translation invariant filter of
  subsets of $X$ and, if $f$ is a function defined on $X/Y$ with values
  in a topological space $\cb$, then $\lim_{z\to\alpha}f(z)=b$ if, and
  only if, $\lim_{x\to\check\alpha}f\circ\pi_Y(x)=b$. It is then clear
  that the condition (ii) above is equivalent to the fact that
  $\slim_{x\to\check\alpha} T_x^*AT_x$ exists for each
  $\alpha\in\SS_{X/Y}$. Now the proof is essentially a repetition of
  the proof of Theorem \ref{th:sphalg}, the filter $\tilde\alpha$ on
  $X/Y$ being replaced by the translation invariant filter
  $\check\alpha$ on $X$.
\end{proof}

There is no simple analogue of Theorem \ref{th:hconj} in the present
context, but one can extend Proposition \ref{pr:ess1} and Theorem
\ref{th:ess2}. Indeed, both Theorems \ref{th:ess3-intro} and
\ref{th:ess4-intro} follow from Theorems \ref{th:recall0},
\ref{th:recall} and \ref{th:sphalgY}. To prove Theorem
\ref{th:ess4-intro} for example, let us set
$\jap{p}=(1+|p|^2)^{1/2}$. Since we have $1+h(p)\sim\jap{p}^{2s}$, it
suffices to prove that for each $Y$ the operator
$\jap{p}^{-2s}V_Y\jap{p}^{-s}$ is in $\maCXY\rtimes X$.  This clearly
follows from Theorem \ref{th:sphalgY}.

\def\cprime{$'$} \def\cprime{$'$}

\end{document}